\documentclass[11pt,a4paper]{amsart}

\usepackage[smalltableaux, centertableaux]{ytableau}
\usepackage{relsize}
\usepackage[foot]{amsaddr}
\usepackage{amsfonts,epsfig,stmaryrd}
\usepackage{newlfont,amsfonts,amssymb,amsmath,amsthm,amsgen,amscd,datetime,dsfont,hyperref,tensor}

\usepackage[small,nohug,heads=littlevee]{diagrams}
\usepackage[normalem]{ulem}
\diagramstyle[labelstyle=\scriptstyle]

\usepackage{multicol,picinpar,enumerate,cleveref,mathabx}
\DeclareMathOperator{\Wedge}{{\mathlarger{\mathlarger{\mathlarger{\wedge}}}}}

\usepackage{euscript,color}

\DeclareMathOperator{\Span}{span}

\DeclareMathOperator{\Id}{Id}

\newcommand{\vf}{\varphi}

\newcommand{\p}{\mathfrak{p}}

\newcommand{\tM}{\widetilde{M}}
\newcommand{\tg}{\widetilde{\g}}
\renewcommand{\th}{\widetilde{\h}}
\newcommand{\tnab}{\widetilde{\nabla}}
\newcommand{\tR}{\widetilde{R}}
\newcommand{\tRE}{\widetilde{R}^{\E}}

\newcommand{\g}{\mathfrak{g}}

\newcommand{\h}{\mathfrak{h}}
\newcommand{\z}{\mathfrak{z}}
\newcommand{\s}{\mathfrak{s}}
\newcommand{\m}{\mathfrak{m}}

\renewcommand{\sl}{\mathfrak{sl}}
\newcommand{\hol}{\mathfrak{hol}}
\newcommand{\stab}{\mathfrak{stab}}
\newcommand{\Hol}{\mathrm{Hol}}
\newcommand{\+}{\oplus}

\newcommand{\so}{\mathfrak{so}}

\newcommand{\co}{\mathfrak{co}}

\newcommand{\gl}{\mathfrak{gl}}

\newcommand{\SO}{\mathbf{SO}}
\renewcommand{\O}{\mathbf{O}}
\newcommand{\SL}{\mathbf{SL}}

\renewcommand{\span}{\mathrm{span}}
\newcommand{\hook}{\makebox[7pt]{\rule{6pt}{.3pt}\rule{.3pt}{5pt}}\,}

\newcommand{\I}{\mathrm{I}}

\newcommand{\1}{\mathbf{1}}
\newcommand{\e}{\mathrm{e}}
\renewcommand{\d}{\mathrm{d}}

\newcommand{\End}{\mathrm{End}}
\newcommand{\R}{\mathbb{R}} %Reals
\newcommand{\N}{\mathbb{N}} %Natural
\newcommand{\E}{\mathcal{E}}

\newcommand{\gE}{g^{\mathcal{E}}}
\newcommand{\tnabE}{\tnab^{\mathcal{E}}}
\newcommand{\nabE}{\nabla^{\mathcal{E}}}

\renewcommand{\H}{H}

\newcommand{\<}{\langle}
\renewcommand{\>}{\rangle}

\theoremstyle{definition}
\newtheorem{definition}{Definition}[section]
\newtheorem{remark}[definition]{Remark}
\newtheorem{example}[definition]{Example}
\newtheorem*{example*}{Example}
\newtheorem*{remark*}{Remark}

\theoremstyle{plain}
\newtheorem{lemma}[definition]{Lemma}
\newtheorem*{lem*}{Lemma}
\newtheorem{proposition}[definition]{Proposition}
\newtheorem{corollary}[definition]{Corollary}
\newtheorem{theorem}[definition]{Theorem}
\newtheorem*{theorem*}{Theorem}
\newtheorem{conjecture}[definition]{Conjecture}
\newtheorem*{conjecture*}{Conjecture}

\numberwithin{equation}{section}

\begin{document}

\title[Lorentzian  homogeneous structures with indecomposable holonomy]{Lorentzian  homogeneous structures\\[1mm]with indecomposable holonomy}

%\date{ {\bf  \today \ at   \xxivtime}}
\thanks{This work was supported by 
 the Australian Research
Council (Discovery Program DP190102360).}
%\\
%Declarations of interest: none}
 \author[Steven Greenwood]{Steven Greenwood}
% \address[Federico Costanza]{...}
\email{steven.greenwood@adelaide.edu.au}

\author[Thomas Leistner]{Thomas Leistner}\address{School of Computer and Mathematical Sciences, University of Adelaide, SA~5005, Australia}\email{thomas.leistner@adelaide.edu.au}

%\dedicatory{\small Dedicated to the memory of Joseph A.~Wolf
%\\
%We were  lucky to have enjoyed Joe’s company, his wisdom, and his with during  two weeks at the Matrix, shortly before his passing.  At the workshop, he gave the last talk, and, as always, he inspired everyone. 
%%It is heartbreaking to know that the e-mail he wanted to send us about a comment he made will never come …
%}

%\thanks{
%The last author was supported by the Australian Research Council via the fellowship FT110100429 and by a Start-Up-Grant of the Faculty of
%Engineering, Computer and Mathematical Sciences of the University of
%Adelaide.
%}
\subjclass[2010]{Primary 
53C30;  Secondary 53C29, 53C50, 53B30}
\keywords{Lorentzian homogeneous spaces, homogeneous structures, Ambrose--Singer connections, connections with torsion,  symmetric spaces, 2-symmetric spaces, holonomy groups} 

\begin{abstract}
For a Lorentzian homogeneous space, we study how algebraic conditions on the isotropy group affect the geometry and curvature of the homogeneous space. More specifically, we prove that a Lorentzian locally homogeneous space is locally isometric to a plane wave if it admits an Ambrose--Singer connection with indecomposable, non-irreducible holonomy. This generalises several existing results that require a certain algebraic type of the torsion of the Ambrose--Singer connection and moreover is in analogy to the fact that a Lorentzian homogeneous space with irreducible isotropy has constant sectional curvature. In addition, we prove results about Lorentzian connections with parallel torsion and for $2$-symmetric connections.
 \end{abstract}

\maketitle

%\tableofcontents

\section{Introduction}
A semi-Riemannian manifold $(M,g)$ is {\em homogenous} if it admits a group of isometries that acts transitively on $M$. Given a transitive group of isometries $G$ and a point  $o\in M$, the {\em isotropy group} is the subgroup of $G$ that fixes $o$. Up to conjugation in $G$, the isotropy in $G$ does not depend on the point, but of course, it depends on  $G$. 
Since $G$ acts transitively, $M$ is diffeomorphic to the quotient $G/H$, where $H$ is the isotropy of $o$.  Moreover, by fixing an semi-orthonormal  basis of $T_oM$, the isotropy $H$ is represented as a subgroup of $\O(t,s)$, where $(t,s)$ is the signature of the metric $g$, by $H\ni \phi\mapsto \d \phi|_o$ and by fixing a basis in $T_oM$.  This is called the {\em isotropy representation of $H$}.
A homogeneous space is {\em reductive} if there is a transitive group of isometries $G$ with isotropy $H$ and in the Lie algebra $\g$ of $G$ there is an $\mathrm{Ad}_H$-invariant complement $\m$ to the Lie algebra $\h$ of $H$. Riemannian homogeneous spaces are always reductive.
%, i.e.~$\g=\h\+\m$ and $Ad(h)\m\subset \m$ for all $h\in H$.

One may ask to which extent algebraic properties of the isotropy group, represented in $\O(t,s)$, determine the geometry of $(M,g)$, for example when $(M,g)$ has a large  isotropy group.  In the case when `large' means `irreducible as a subgroup of $\O(t,s)$', the situation is rather flexible in the Riemannian setting, see  \cite[Chapter 7]{besse87} for an overview, but quite rigid for indefinite metrics, where often irreducibility  forces the homogeneous space to be locally symmetric,~see for example \cite{CortesMeinke17,CortesMeinke18}. The Lorentzian case is particularly rigid, as the following remarkable result shows.
\begin{theorem}[Zeghib \cite{zeghib04}]
\label{zeghibtheo}
If a Lorentzian homogeneous space of dimension $m\ge 3$ admits an irreducible isotropy group, then it has constant sectional curvature.
%\footnote{We believe that the assumption on the dimension is not needed since $\O(1,1)$ is not irreducible.}.
\end{theorem}
The result of the theorem can also be deduced from the algebraic fact that a subgroup of $\O(1,m-1)$ that acts irreducibly on $\R^{1,m-1}$ contains the connected component of $\O(1,m-1)$, see \cite{benoist-harpe03,olmos-discala01,att05}. Zeghib however has provided a  beautiful geometric proof that is based on the fundamental geometric principle that non-compactness of the isotropy group leads to totally geodesic hypersurfaces,  \cite{zeghib04}, see also \cite{AroucheDeffafZeghib07} and our Section~\ref{zeghibsec}. Contemplating this rigidity, one may relax the condition of irreducibility to the following condition that is more natural in the indefinite context. We say that a subgroup $H\subseteq \O(1,m-1)$ (or a subalgebra $\h\subseteq \so(1,m-1)$) is {\em indecomposable}\footnote{By some authors  this property is called {\em weakly irreducible}, e.g.~in  \cite{olmos-discala01,ErnstGalaev22,ErnstGalaev23}.} if  there is no non-degenerate subspace of $\R^{1,m-1}$ that is invariant under $H$ (or $\h$). The name {\em indecomposable} and the importance of the concept comes from its equivalence, when $m>2$, to the  property that $\R^{1,m-1}$ does not decompose  into a non-trivial sum of subspaces that are invariant.
By the above comments, a proper indecomposable subgroup of the connected component of $\O(1,m-1)$ must admit an invariant null line in $\R^{1,m-1}$. Here, and throughout, by {\em null} we mean light-like, or isotropic.

To our knowledge, all  examples of reductive Lorentzian homogeneous spaces of dimension $m\ge 4$ with indecomposable, non-irreducible isotropy are  homogeneous {\em plane waves}, a special class of Lorentzian manifolds. A  {\em plane wave} is a Lorentzian manifold that 
\begin{enumerate}
\item admits a parallel null vector field $\xi$,   
\item  its curvature tensor  satisfies  $R(X,Y)=0$ for all $X,Y\in \xi^\perp$, and
\item
 $\nabla_XR=0$ for all $X\in \xi^\perp$.   
\end{enumerate}A Lorentzian manifold with only (1) and (2) is called a {\em pp-wave} (`plane fronted with parallel rays'). 
In dimension $3$ every Lorentzian manifold with parallel null vector field is a pp-wave. 
It has been shown in \cite{GlobkeLeistner16} that under a mild non-degeneracy condition a homogeneous pp-wave of dimension $m\ge 4$  is a plane wave. Moreover, homogeneous plane waves are reductive and have been classified in \cite{blau-oloughlin03}.  They admit an indecomposable isotropy group  $\R^n\subset \SO(1,n+1)$, where $m=n+2$ is the dimension of $M$,  and hence have the same isotropy as Cahen--Wallach spaces.  Cahen and Wallach have shown in  \cite{cahen-wallach70} that a Lorentzian symmetric space with indecomposable holonomy  either has constant sectional curvature or is universally covered by a Cahen--Wallach space, which is a locally symmetric plane wave (see Section~\ref{homplanesec} for their definition).
This scarcity of examples and the analogy with symmetric spaces leads us to conjecture the following.
\begin{conjecture}
\label{conj}
A reductive Lorentzian   homogeneous space $G/H$ of dimension $m\ge 4$  with indecomposable, non-irreducible isotropy $H$ is a plane wave.
\end{conjecture}
This conjecture appears to be very optimistic when considering the following equivalent statement: the isometry group of a Lorentzian homogeneous space that is not  a plane wave does not have any transitively acting subgroups with indecomposable, non-irreducible isotropy. Nevertheless, several results towards this conjecture have been proven for   special situations in  \cite{Meessen06,Montesinos-Amil01, ErnstGalaev22, ErnstGalaev23}. 
The approach in these papers is based on the existence of a {\em homogeneous structure}, or equivalently, an  {\em Ambrose--Singer connection}, i.e.~a  connection $\tnab$ that is compatible with the metric $g$ and has $\tnab$-parallel torsion and curvature.
The existence of such a connection is equivalent to `local reductive homogeneity',  which,  under some global assumptions, is equivalent to reductive homogeneity; see \cite{AmbroseSinger58,tricerri-vanhecke83,GadeaOubina92}, and our Section~\ref{review}. An example of an Ambrose--Singer connection is the canonical connection of a reductive homogeneous space $G/H$. The holonomy of this connection is contained in the isotropy representation of $H$.
In the current paper we will make the following substantial step towards Conjecture~\ref{conj}.
% that is also based on the holonomy of an Ambrose--Singer connection, but does not require the assumptions in \cite{Meessen06,Montesinos-Amil01, ErnstGalaev22, ErnstGalaev23}. We will prove 
% the following result.
\begin{theorem}\label{maintheo}
Let $(M,g)$ be a Lorentzian manifold of dimension $m\ge 4$  that admits an Ambrose--Singer connection with indecomposable, non-irreducible holonomy. Then the universal cover of $(M,g)$ is a locally homogeneous plane-wave.
\end{theorem}
We show by  example that the dimension restriction $m\ge 4$ is sharp, in this theorem as well as in the conjecture: in Example~\ref{counterex} we construct  
a left-invariant metric $g$ 
on the universal cover of $\SL(2,\R)$ that is not a pp-wave metric, together with an Ambrose--Singer connection for $g$ with indecomposable, non-irreducible holonomy algebra $\R\subset \SO(1,2)$, which can be represented as a homogeneous space $\R\ltimes \widetilde{\SL}(2,\R)/\R$, where $\R\subset \SO(1,2)$ is indecomposable.

We should elaborate on the difference between Conjecture~\ref{conj} and Theorem~\ref{maintheo}. 
On the one hand, under some global requirements, the assumptions of the theorem  imply the assumptions in the conjecture, as the existence of an Ambrose--Singer connection with holonomy $H$ allows to realise  $(M,g)$ as a homogeneous space $G/H$, see our Section~\ref{review} for details and references.  On the other hand, if $M=G/H$ is a reductive homogeneous space with  $H\subseteq \O(1,m-1)$, then the canonical connection of $G/H$ is an Ambrose--Singer connection. If $\widetilde{H}$ denotes its holonomy, then $\widetilde{H}\subseteq H$ and $M=\widetilde{G}/\widetilde{H}$ for a group $\widetilde{G}\subseteq G$, the {\em transvection group}, again see our Section~\ref{review}. In general,  indecomposibility may be lost when passing to the subgroup $\widetilde{H}\subseteq H$, and hence it is not clear if the assumptions in Conjecture~\ref{conj}  imply the assumptions in Theorem~\ref{maintheo}.  Nevertheless, Theorem~\ref{maintheo} is a crucial step towards the validity of the conjecture.

Our  Theorem~\ref{maintheo} is also substantial generalisation of results in 
 \cite{Meessen06,Montesinos-Amil01, ErnstGalaev22, ErnstGalaev23}, which we will briefly review  here. Since an Ambrose--Singer connection preserves the metric, it is completely determined by its torsion $T$, which by definition is parallel, so its algebraic type does not change. The torsion 
 %is an element in the bundle $\Lambda^2T^*M\otimes TM$ and, using the metric, its fibres  are isomorphic to 
 can be identified with an element in $\Lambda^2(\R^m)^* \otimes (\R^m)^*$, which decomposes into three irreducible modules of $\so(1,m-1)$. One is given by the $3$-forms, whose elements correspond to  {\em totally skew torsion}. An invariant complement to the $3$-forms is given by those $T_{abc}\in \Lambda^2(\R^m)^* \otimes (\R^m)^*$ that satisfy $T_{[abc]}=0$, where the brackets indicate the skew-symmetrisation over enclosed indices. This module splits further into two irreducible modules, the trace-free tensors and those that are pure trace, and the corresponding torsion is called {\em twistorial} and {\em vectorial}, respectively.  Vectorial torsion is called {\em degenerate} if the $1$-form $T_{ab}{}^b$ is null.
Under  assumptions on the algebraic type of the torsion $T$ of an Lorentzian Ambrose--Singer connection, the following results have been shown:
\begin{enumerate}
\item If $T$ is  vectorial, then $(M,g)$ is a singular homogeneous plane wave in the case when  $T$ is degenerate \cite{Montesinos-Amil01},
or  has constant sectional curvature otherwise
 \cite{GadeaOubina97}.
\item If $T$ is twistor-free (i.e.~no twistorial part) and has  degenerate, nonzero vectorial part, then $(M,g)$ is a singular homogeneous plane wave  \cite{Meessen06}.
\item If $T$ is totally skew and the holonomy of  $\tnab$ is indecomposable, then $(M,g)$ is a regular homogeneous plane wave  \cite{ErnstGalaev22}.
%, ErnstGalaev23}.
\end{enumerate}
Note that  twistorial torsion lies in the   largest of the irreducible torsion modules, but has not yet been dealt with in the literature and is considered to be the most difficult case. Since our result in Theorem~\ref{maintheo} does not require {\em any} assumptions on the algebraic type of the torsion, it  also covers this case.
 
 The main ingredients for the proof of Theorem~\ref{maintheo} are algebraic and provided in Section~\ref{algsec}. They consist of a description of submodules of the torsion module and of the module of algebraic curvature tensors (with torsion) on which an indecomposable subalgebra $\h\subsetneq\so(1,m-1)$ acts trivially.  This imposes very strong conditions on the parallel torsion and curvature of an Ambrose--Singer connection that allow us to prove Theorem~\ref{maintheo}. En route to Theorem~\ref{maintheo}, in Section~\ref{LASCsec} we prove two key results about Lorentzian connections with torsion and $2$-symmetric connections. By a {\em Lorentzian connection} we mean a connection on a Lorentzian manifold $(M,g)$ such that $\nabla g=0$. If such a connection has indecomposable, non-irreducible holonomy, it admits a parallel subbundle of null lines in $TM$ and hence locally a recurrent null vector field. We can summarise our results about Lorentzian connections as follows.
\begin{theorem}
Let $(M,g)$ be a simply connected Lorentzian manifold of dimension $\ge 4$ and $\tnab$ be a Lorentzian connection with indecomposable, non-irreducible holonomy. 
\begin{enumerate}
\item If $\tnab$ is  2-symmetric, i.e.~if its curvature $\tR$ satisfies $\tnab\tnab \tR=0$, then it admits a parallel null vector field (see Theorem~\ref{VFtheo}).
\item If $\tnab$ has parallel torsion, then it admits a parallel null vector field if and only if the Levi-Civita connection of $(M,g)$ admits a parallel null vector field, and both vector fields span the same null line bundle. 

Moreover, if $\tnab$ does not admit a parallel null vector field, then it torsion-free and hence  equal to the Levi-Civita connection (see Theorem~\ref{vftheo}).
\end{enumerate}
\end{theorem}
The first result also holds in dimension $3$ for locally symmetric connections, i.e.~with $\tnab\tR=0$. These results generalise the results 
some results in \cite{senovilla08} and \cite{AlekseevskyGalaev11} for second order symmetric spaces, and 
in \cite{ErnstGalaev22,ErnstGalaev23} for Lorentzian connections with parallel totally {\em skew} torsion. We believe that the latter will be of interest in the study of Lorentzian connection with torsion in supergravity and string theory.

\subsubsection*{Acknowledgements} We would like to thank Michael Eastwood for helpful discussions and invaluable inspiration. TL also  thanks Ian Anderson for some help with Maple.
\section{Preliminaries}
\subsection{Connections} 
If $\nabla$ is an affine connection on a smooth manifold $M$, i.e.~a connection on $TM$, we define its torsion tensor and curvature tensor as  
\[T^\nabla(X,Y) =\nabla_XY-\nabla_YX-[X,Y],\qquad R^\nabla(X,Y)=[\nabla_X,\nabla_Y]-\nabla_{[X,Y]}.\]
We say that a connection is {\em locally symmetric} if $R^\nabla $ is parallel, i.e.~$\nabla R^\nabla=0$. If $g$ is a semi-Riemannian metric, we say that $\nabla$  {\em preserves $g$} if $\nabla g=0$. If two connections $\nabla$ and $\tnab$ preserve the same metric $g$, then $S$ defined by \begin{equation}\label{Sdef}
S(X,Y)=\nabla_XY-\tnab_XY
\end{equation}
 is a section of $T^*M\otimes \so(TM,g)$, where $\so(TM,g):=\{A\in \End(TM)\mid A\cdot g=0\}$ in which the $\cdot$ denotes the natural action of  endomorphisms on tensors. For $S\in T^*M\otimes \so(TM,g)$ and $X\in TM$ we will also write $S(X):=S(X,.)\in \so(TM,g)$ throughout.

\begin{lemma}\label{curvlem}
Let $\nabla$ be a torsion-free connection and $\tnab$ an affine connection with torsion $T$. 
\begin{enumerate}
\item\label{curvlem1} 
 If $S(X,Y)$ is defined by~(\ref{Sdef}), then the curvatures of both connections are related by
\[R(X,Y)=\tR(X,Y)+(\tnab_XS)(Y)-(\tnab_YS)(X)+S(T(X,Y))+[ S(X),S(Y)].\]
%\item\label{curvlem2}
%If in addition  $\tnab S=0$, then 
%\[
%R(X,Y)=\tR(X,Y)+S(T(X,Y))+[ S(X),S(Y)],
%\]
\item\label{curvlem2}
If  both connections preserve  a semi-Riemannian metric $g$  and $\tnab $ has parallel torsion, i.e. $\tnab T=0$, then
\[ \tR(X,Y)Z+\tR(Y,Z)X+\tR(Z,X)Y= T(T(X,Y),Z)+ T(T(Y,Z),X)+T(T(Z,X),Y).\]
\item\label{curvlem3} If in addition to (\ref{curvlem2}), the conncetion $\tnab$ is is locally symmetric, then $\tnab R=0$ and $\nabla_XR=S(X)\cdot R$.
\end{enumerate}
\end{lemma} 
\begin{proof}
(\ref{curvlem1}) By the definiton of curvature and  with $\nabla$ torsion-free we get
\[R(X,Y)=\tR(X,Y )+[ \tnab_X,S(Y)] -[ \tnab_Y,S(X)] -S([X,Y] + [S(X),S(Y)].\]
Then the statement follows from $[X,Y]=\tnab_XY  -\tnab_YX-T(X,Y)$.

(\ref{curvlem2}) Let $\nabla$ be the Levi--Civita connection of a semi-Riemannian metric $g$ and $\tnab g=0$. Because of the isomorphism $\Lambda^2\otimes TM\simeq T^*M\otimes \so(TM,g)$, the section $S$ of $T^*M\otimes \so(TM,g)$ is completely determined by the torsion $T(X,Y)=-S(X,Y)+S(Y,X)$ as 
\[g(S(X,Y),Z)= -\tfrac{1}{2}\left( g(T(X,Y),Z)+ g(T(Z,Y),X)+g(T(Z,X),Y)\right).
\] 
Hence, if $\tnab T=0$ we also have $\tnab S=0$, so that the formula in (\ref{curvlem1}) simplifies accordingly. Then writing out the cyclic sum and using that $T(X,Y)=-S(X,Y)+S(Y,X)$ yields the formula for Bianchi symmetry.

(\ref{curvlem3}) 
If in addition $\tnab\tR=0$,  the tensor fields of $\Lambda^2T^*M \otimes \End(TM)$ defined by
\[S\circ T(X,Y):=S(T(X,Y)),\quad [S,S](X,Y):=[S(X),S(Y)] \]
are both $\tnab$-parallel. Indeed, with $S$ and $T$ is parallel, 
\[(\tnab_X (S\circ T)) (X,Y)= S((\tnab_X T)(Y,Z))=0,\]
and 
\[(\tnab_X S\circ T)(Y,Z)
=
\left[\tnab_X,[S(Y),S(Z)]\right]
-
\left[[ \tnab_X,S(Y)],S(Z)]\right]
-
\left[S(Y),[ \tnab_X,S(Z)]]\right]
=0,
\]
by the Jacobi identity for endomorphisms. Hence, $\tnab R=\tnab \tR=0$.
\end{proof}

\subsection{Reductive  homogeneous spaces and Ambrose--Singer connections}\label{review}
Let $(M,g)$ be a semi-Riemannian manifold with Levi-Civita connection $\nabla$ and with curvature tensor $R$. A {\em homogeneous structure} for $(M,g)$ is a section $S$ of $T^*M\otimes \so(TM,g)$ such that
\[\nabla_XS=
S(X)\cdot S,\quad \nabla_XR=S(X)\cdot R,\]
for all $X \in TM$. 
The existence of  homogeneous structure $S$ is equivalent to the existence of an {\em Ambrose--Singer connection}, i.e.~a connection $\tnab$  that is locally symmetric, has parallel torsion, and preserves the metric $g$.
%\begin{equation}\label{ASdef}\tnab g=0,\quad \tnab \tR=0,\quad \tnab T=0,\end{equation}
%where $\tR$ is the curvature of $\tnab$ and $T$ its torsion, i.e.\[T(X,Y) =\tnab_XY-\tnab_YX-[X,Y].\] 
Since $\tnab g=0$, the connection $\tnab$ is uniquely determined by its torsion. Hence, the relation between an Ambrose--Singer connection $\tnab $ and a homogeneous structure $S$ is given by~(\ref{Sdef}).
Consequently, the condition $\tnab T=0$ is equivalent to $\tnab S=0$ and 
by Lemma~\ref{curvlem} a connection $\tnab$ is an Ambrose--Singer connection if and only if
\begin{equation}\label{ASdef1}\tnab g=0,\quad \tnab R=0,\quad \tnab S=0,\end{equation}
where $S$ defined in~(\ref{Sdef}) and $R$ is the  curvature tensor of the Levi-Civita connection.

The name suggests that there is a close relation between homogeneous structures and homogeneous spaces, which is indeed the case:
\begin{enumerate}
\item[(A)]
Given a  reductive homogeneous space $M=G/H$, with $G$ a subset of the isometry group of $(M,g)$ and $H$ the isotropy group at a point $o=eH$, there is an Ambrose--Singer connection.
\end{enumerate}
This Ambrose--Singer connection in (A) is the canonical connection of the reductive homogeneous space, see \cite{AmbroseSinger58} or \cite[Section II.11]{ko-no1} for 
details.
%
%is canonically defined as follows,
%The Lie algebra of $G$ is anti-isomorphic to a subalgebra $\g$ of the Lie algebra of  Killing vector fields of $(M,g)$, and $\h$ is the subalgebra of Killing vector fields in $\g$ that vanish at $o$.  For a reductive homogeneous space, i.e.~ when 
% $\g$ decomposes $\mathrm{Ad}_H$-invariantly as $\g=\h\+\m$, one can define the {\em canonical connection of $G/H$} by the properties $\tnab g=0$, \Tcom{not sure about the signs here.} $T(X,Y)=-\mathrm{pr}_\m([X,Y])$ and  $\tR(X,Y)Z=-[\mathrm{pr}_{\h}([X,Y]),Z]$ for Killing vector fields $X,Y,Z$ from $\m$. The canonical connection defined in this way  is an Ambrose--Singer connection.  
More importantly, in a specific sense the converse of  statement (A) holds:
 \begin{enumerate}
\item[(B)]
 If a geodesically complete, simply connected semi-Riemannian manifold $(M,g)$ admits a homogeneous structure (or equivalently, an Ambrose--Singer connection), then $(M,g)$ is a reductive homogeneous space.
 \end{enumerate}
  This was first proved by Ambrose and Singer in \cite[Section 3]{AmbroseSinger58} for Riemannian manifolds, see also \cite{tricerri-vanhecke83}. It was later generalised to indefinite metrics in \cite{GadeaOubina92}\footnote{We believe that there is a gap in the proof of \cite[Theorem 1]{GadeaOubina92}. It appears that homogeneous plane waves (see Section~\ref{homplanesec}) provide a counterexample to the equivalence of (b) and (c), ibid: they are reductive but their isotropy is null for the Killing form.
  }
  , see also \cite[Theorem I.40]{Kowalski80} for the more general affine setting, and \cite[Theorem 2.2.1]{CalvarusoCastrillon-Lopez19}.
  
There is a local version of the statement in (B). A semi-Riemannian manifold $(M,g)$ is {\em locally homogeneous} if for each pair  of points  $p$ and $q$ in $M$ there is are neighbourhoods $U$ and $V$ of $p$ and $q$ and an isometry from $U$ to $V$ that maps $p$ to $q$, see for example \cite[Section 7.3]{ambra-gromov91} or \cite[Definition 1.4.6]{CalvarusoCastrillon-Lopez19}. This is for example the case if every point $p\in M$ admits a neighbourhood $U$ of $p$ and Killing vector fields on $U$ that span $T_pM$. In \cite[Section 3.1]{CalvarusoCastrillon-Lopez19} it is shown that reductivity for locally homogeneous semi-Riemannian manifolds makes sense and that
 \begin{enumerate}
\item[(C)] a semi-Riemannian manifold $(M,g)$ admits an Ambrose--Singer connection if and only if $(M,g)$ is reductive locally homogeneous, \cite[Theorems 3.1.14 and 3.1.15]{CalvarusoCastrillon-Lopez19}.
\end{enumerate}
Finally, we recall the relation between the isotropy algebra $\h$ of a homogeneous space and the holonomy algebra of the canonical connection, as it can be found in \cite[Theorems I.25 and I.33]{Kowalski80} for affine homogeneous spaces.
If
 $M=G/H$ is a semi-Riemannian reductive homogeneous space with reductive Lie algebra decomposition $\g=\h\+\m$, then the connected Lie group $\widetilde{G}$ that corresponds to the ideal $\widetilde{\g}:=[\m,\m]+\m=\mathrm{pr}_\h([\m,\m])\+\m$ in $\g$ acts transitively on $M$ and  $\th:=\mathrm{pr}_\h([\m,\m])$ is equal to the holonomy algebra of the canonical connection.  
From this it follows that every homogeneous space $G/H$ has a representation as $\widetilde{G}/\widetilde{H}$,  where $\th$ is the holonomy algebra of the canonical connection of $G/H$. The group $\widetilde{G}$ is  called the transvection group of the homogeneous space $G/H$ and the presentation of $M$ as $\widetilde{G}/\widetilde{H}$ the {\em transvection} or {\em prime presentation}, 
\cite[page 41]{Kowalski80}. 

By  (C), if $\tnab$ is an Ambrose--Singer connection on $(M,g)$, then $(M,g)$ is locally homogeneous. If $T$ and  $\tR$ are the torsion and curvature of $\tnab$, $o\in M$ and $\m:=T_oM$, then by the Ambrose--Singer holonomy theorem in \cite{as}, 
the  holonomy algebra $\th$ at $o\in M$
is  $\th  =\span\{\tR|_o(X,Y)\mid  X,Y\in \m\}$. 
One can define a Lie bracket on $\tg:=\th\+\m$ by extending the Lie bracket of $\th\subseteq \so(\m)$ by
\[[ H,X]:=H(X),\qquad [ X,Y ]:=-\tR|_o (X,Y)-T|_o(X,Y).\]
Then there is a unique simply connected Lie group $\widetilde{G}$ with Lie algebra $\tg$ and the unique connected subgroup $\widetilde{H}$ with Lie algebra $\th$. If $\widetilde{H}$ is closed in $\widetilde{G}$, the homogeneous space $\widetilde{G}/\widetilde{H}$ carries the metric coming from $g|_o$ on $\m$ and  $(M,g)$ is locally isometric to $\widetilde{G}/\widetilde{H}$. The data $(\m, \tR|_o, T|_o)$ is  called an {\em infinitesimal model} of the locally homogeneous space $(M,g)$.  
%
%{\tt 
%\Tcom{Questions!}
%``the'' or ``a'' --- is the infinitesimal model unique?
%
%What happens for a Lie group with left invariant metric. Isn't here an  AS connection that is flat?} 

In the light of these facts,  Theorem~\ref{maintheo} can also be stated as: {\em A  Lorentzian reductive locally homogeneous space is a plane wave if it admits  an  infinitesimal model that has indecomposable, non-irreducible  isotropy.}

\subsection{Lorentzian homogeneous spaces}

We will not give a comprehensive review about what is known for Lorentzian homogeneous spaces and homogeneous structures --- for this we refer to \cite{CalvarusoCastrillon-Lopez19} and the references therein --- but we will focus on some aspects that are relevant for the context of  Conjecture~\ref{conj} and Theorem~\ref{maintheo} in the introduction.

\subsubsection{Zeghib's approach}
\label{zeghibsec}
Here we will see how far Zeghib's approach gets us towards establishing  Conjecture~\ref{conj}.  The proof of Therorem~\ref{zeghibtheo}  in \cite{zeghib04} is based on the following principle.
\begin{proposition}\label{zeghibprop}
Let $(M,g)$ is a Lorentzian homogeneous space and $o\in M$ and $H$ the isotropy of $o$. If the image under the isotropy representation $H\to \O(T_oM,g|_o)$ is not precompact, then there is a totally geodesic null hypersurface through $o$.
\end{proposition}
The proof of this in \cite[Fact 6]{zeghib04}, see also \cite{steven-thesis} for more details, proceeds by considering  a sequence $(h_n)_{n\in\N}$ in $H$ such that  $\d h_n|_o\in \O(T_oM,g|_o)$ diverges. Let $V_n$ be the graph of $\d h_n |_o$ in $T_o M\oplus T_o M$. Since the Grassmannian of $m$-planes in $\R^{2m}$ is compact, $(V_n)_{n\in \N}$ has a convergent subsequence with limit $V$, which is no longer a graph as otherwise $\d h_n|_o$ could not diverge.
Then all $E_n=  \exp_o(V_n)=\mathrm{graph}(f_n)$ are totally geodesic and totally null  for the metric $(g,-g)$ on $M\times M$, and one can show that the same holds for $E= \exp_o(V)$. 
Finally, since $V$ is not a graph, the intersection of $V$ with $\{0\}\times T_oM$ is not trivial and hence must be of dimension $1$, because $V$ is totally null and $g$ is Lorentzian. Therefore, the projection of $V$ onto  $T_oM\times \{0\}$ is a null hyperplane and $E\cap M\times \{o\}$ is a  totally geodesic hypersurface.

In order to proceed with the proof of Theorem~\ref{zeghibtheo}
one considers the subspace of $T_oM$,
\[ W:=\mathrm{span}\{v\in T_oM\mid v\text{ is null and $\exp_o(v^\perp)$ is a totally geodesic hypersurface}\},\]
which is invariant under the isotropy representation. In the case when $H$ is irreducible (and hence not precompact),  $W$ is non-trivial  by Proposition~\ref{zeghibprop}, and hence must be all of $T_oM$. Then one can show, \cite[Proposition 3]{zeghib04} that this implies constant sectional curvature. 

In the case when the isotropy $H$ of a homogeneous Lorentzian space $(M,g)$ is not irreducible, but indecomposable (and hence not precompact),  $W$ must be degenerate, and thus defines a null line 
$L=W\cap W^\perp$ that is invariant under the isotropy and such  that $\exp_o(L^\perp )$ is a totally geodesic hypersurface. Since $(M,g)$ is homogeneous and since $L$ is isotropy invariant,  this defines a null line and  a vector distribution $\mathcal L\subset \mathcal L^\perp$, whose leaves are totally geodesic.  
We arrive at the following conclusion.
\begin{proposition}\label{zeghibindecprop}
Let $(M,g)$ is a Lorentzian homogeneous space $M=G/H$ with indecomposable, non-irreducible $H$. Then there is a $G$-invariant hyperplane distribution $\mathcal H\subset TM$ that is
auto-parallel, i.e.~$\nabla_\xi\eta\in \Gamma(\mathcal H)$ for all sections $\xi$ and $\eta$ of $\mathcal H$, and hence defines a foliation into $G$-invariant totally geodesic hypersurfaces.
\end{proposition}
Being auto-parallel  is a strictly weaker condition for a vector distribution than being  parallel, which is would be the first step towards proving that $(M,g)$ is a plane wave. We do not know how to proceed from here directly towards Conjecture~\ref{conj}, and hence we will follow the approach of homogeneous structures and Ambrose--Singer connections.

\subsubsection{Homogeneous plane waves}
\label{homplanesec}
A Lorentzian manifold $(M,g)$ with Levi-Civita connection $\nabla$ and curvature tensor $R$ is called a {\em pp-wave (plane fronted with parallel rays)} if it admits a parallel $\nabla$-parallel null vector field $\xi$ and obeys the curvature condition
\begin{equation}
\label{ppcurv}
R(X,Y)=0,\quad\text{ for all }X,Y\in \xi^\perp=\{Z \in TM\mid g(X,\xi)=0\}.\end{equation}
Because of the symmetries of $R$, this condition is equivalent to the condition that all $R(U,V)$, for $U,V\in TM$, map the vector distribution $\xi^\perp$ to the line bundle $\R\cdot \xi$. Locally, the metric $g$ of an $m=n+2$-dimensional  pp-wave is given as
\begin{equation}\label{pploc}
2\,\d v\, \d t +\delta_{ij}\, \d x^i \d x^j+ h(t,x^1,\ldots, x^n)\, \d t^2 ,\end{equation}
for a function $h$ of the coordinates $t$ and $x^i$, where $i,j,k=1, \ldots, n$. If there is a point where the matrix $\partial_i\partial_j h$ is non-degenerate, then the holonomy group of a pp-wave is indecomposable. Hence, indecomposable pp-waves are examples of Lorentzian manifolds with special holonomy, i.e. the holonomy algebra is reduced from $\so(1,n+1)$ but indecomposable. In fact, their holonomy is given by $\g_-=\R^n$, which is an abelian ideal in the stabiliser of a null vector  in $\so(1,n+1)$, see Section~\ref{algsec1} for details.

A {\em plane wave} is a pp-wave that satisfies the additional curvature condition
\begin{equation}
\label{planecurv}
\nabla_XR=0,\quad\text{ for all }X\in \xi^\perp.\end{equation}
In the local form of the metric this forces the function $h$ to be o the form \begin{equation}
\label{planeloc}
h(t,x^1,\ldots, x^n) =x^iQ_{ij}(t)x^j, \quad\text{ for $Q_{ij}$ a symmetric matrix of functions of  $t$ .}\end{equation}
A plane wave is locally symmetric if and only if the functions $Q_{ij}$ are constant.
Hence, a special class of plane waves are the Lorentzian symmetric spaces with indecomposable, non-irreducible holonomy, the {\em Cahen--Wallach spaces}. They are defined by $M=\R^{n+2}$ and $g$ as in (\ref{pploc}) with $h(t,x^1,\ldots, x^n)= x^i Q_{ij} x^j$ for a symmetric, non-degenerate, constant matrix $Q=(Q_{ij})$.  Cahen and Wallach have shown in \cite{cahen-wallach70,Cahen98} that a Lorentzian symmetric space  either has constant sectional curvature or is universally covered by a Cahen--Wallach space.  The isometry algebra of a Cahen--Wallach space is equal to 
a semi-direct sum of the $(2n+1)$-dimensional Heisenberg algebra with $\R\oplus \z_{\so(n)}(Q)$, where $\z_{\so(n)}(Q)$ is the centraliser in $\so(n)$ of the matrix $Q$. The isotropy algebra is given by $\R^n\rtimes  \z_{\so(n)}(Q)\subset \so(1,n+1)$ and hence indecomposable. 

Homogeneous plane waves other than Cahen--Wallach spaces have been classified in \cite{blau-oloughlin03}. They split into {\em regular} and {\em singular} homogeneous plane waves and are parametrised by  a skew-symmetric constant matrix $F$ and  a symmetric constant matrix $Q_0$. For the regular ones it is $M=\R^{1,n+1}$ and the matrix  $Q=(Q_{ij})$ in~(\ref{planeloc}) is of the form
\[
Q(t)=\exp(t F) Q_0 \exp(-t F),
\]
where $\exp$ denotes the matrix exponential.
A singular homogeneous plane wave is defined on a open set $\{ t>a\}$ of $\R^{n+2}$, for  $a\in \R$, and the matrix  $Q=(Q_{ij})$ in~(\ref{planeloc}) is of the form
\[
Q(t)=\tfrac{1}{(t-a)^2} \exp(\log(t-a)  F) Q_0 \exp(-\log(t-a)  F).
\]
The proof of this result proceeds studying the implications for the metric of the existence of $\ge m$ solutions of  the Killing equation. 
With this classification, our Theorem~\ref{maintheo} implies that a reductive homogeneous Lorentzian manifold, whose canonical connection has indecomposable, non-irreducible holonomy, must  be one of these two types of manifolds. For both types, the isotropy for the isometry algebra contains $\R^n=\g_-\subset \so(1,n_+1)$, and hence is indecomposable and, if the Levi-Civita holonomy is also indecomposable, equal to their holonomy.

Also by analysing the Killing equation, in \cite{GlobkeLeistner16} it has been shown that a locally homogeneous pp-wave $(M,g)$ with indecomposable holonomy and such that the rank of the curvature operator acting on $2$-forms is greater than $1$ must be a locally homogeneous plane wave. The condition on the rank of $R$ is of course not satisfied in dimension $m=3$ and there is an example of a $3$-dimensional, indecomposable locally homogeneous pp-waves that is not a plane wave, \cite[Example 4.1]{GlobkeLeistner16}, where 
\[g=2\d v \d t  +\e^{2 x}\d t ^2 +\d x^2.\]
This example however has a $3$-dimensional Lie algebra of Killing vector fields and hence trivial isotropy, and thus is not in conflict with  Conjecture~\ref{conj}. See the next paragraph  and  Section~\ref{dim3sec} for results in dimension $3$.

\subsubsection{Lorentzian homogeneity in dimension $3$}
\label{dim3sec0}
Theorem~\ref{maintheo} and Conjecture~\ref{conj} do not hold in dimension $3$, see Section~\ref{dim3sec}. Nevertheless, there is a classification of Lorentzian homogeneous spaces in dimension $3$, which we will briefly review here. It has been claimed in \cite[Theorem 3.1]{calvaruso07homog} that a $3$-dimensional Lorentzian homogeneous space is locally isometric to a symmetric space or a Lie group with left-invariant metric. 
This theorem may be missing the assumption of completeness, see \cite[Theorem 1.1]{calvaruso07homog} instead or
\cite[Theorem 4.3.3]{CalvarusoCastrillon-Lopez19}. Indeed,
 it has been noted recently in \cite[Theorem 1.6]{AlloutBelkacemZeghib23}
that there is an incomplete homogeneous plane wave that is neither locally symmetric nor locally isometric to a Lie group with left invariant metric. This metric is defined on $\{t>0\}\subset \R^3$ and is of the form
\[g_b:=2\, \d v\,\d t- \tfrac{ b\,x^2}{t^2} \d t^2 +\d x^2,\quad\text{ for some constant $b>\tfrac{1}{4}$},\]
 see  \cite[Theorem 1.3]{AlloutBelkacemZeghib23}. 
  Locally symmetric spaces are classified by \cite{cahen-wallach70}, see previous section, whereas  $3$-dimensional Lie groups with left-invariant Lorentzian metric have been classified in \cite{Rahmani92}. Our counterexample to Theorem~\ref{maintheo} and Conjecture~\ref{conj}  in Section~\ref{dim3sec} is a left-invariant metric on the universal cover of $\SL(2,\R)$.
%
%\subsection{Lorentzian manifolds with special holonomy, pp-waves and plane waves}\label{ppsec}
%In analogy with  Riemannian manifolds, we say that a Lorentzian manifold of dimension $m$ has {\em special holonomy} if   the holonomy algebra of its Levi-Civita connection is not isomorphic to $\so(1,m-1)$ but acts indecomposably on $T_pM$,~i.e. without non-degenerate invariant subspace. Quite in contrary to Riemannian manifolds, this implies that the holonomy algebra is no longer irreducible. This is due to the fact that $\so(1,m-1)$ has no proper irreducible subalgebras, \cite{benoist-harpe03,olmos-discala01,att05}. It also implies that there $(M,g)$ is nowhere locally isometric to s semi-Riemannian product manifold. 
%

\section{Algebraic results}\label{algsec}

%\section{Indecomposable subalgebras of $\so(1,n+1)$, torsion module and curvature tensors}
\subsection{Algebraic conventions for Minkowski space}
\label{algsec1}
Let $V:=\R^{1,n+1}$ be the $(n+2)$-dimensional Minkowski space with $n\ge 1$, which is assumed from now on,  and denote by $\<.,.\>$  the Minkowski inner product with one negative and $n+1$ positive eigenvalues. We fix a basis $(e_-, e_1, \ldots, e_n, \e_+)$ such that 
\[\I:=\big(\< e_A,e_B\> \big)_{A,B=-, 1\,\ldots, n, +}=\begin{pmatrix} 0&0&1\\ 0&\1_n&0\\1&0&0\end{pmatrix},\]
where here, and from now on capital indices $A,B,\ldots$ run over $-, 1, \ldots, n, +$. Small indices $i,j,k,\ldots$ will run from $1$ ton $n$. 
We call such a basis a {\em Witt basis}. With respect to a Witt basis
we define 
\[V_-:=\R\cdot \e_-,\qquad V_+:=\R\cdot \e_+ ,\qquad V_0:=V_-^\perp\cap V_+^\perp=\Span(e_i)_{i=1,\ldots, n}\simeq \R^n.\]
Note that $\<.,.\> $ restricts on $V_0$ to the standard Euclidean inner product $\<e_i,e_j\>=e_i^\top e_j=\delta_{ij}$. 
We  define a basis of $V^*$ as
\[e^-:=\<e_+,.\>, \quad e^i:=\<e_i,.\>,\qquad e^+:=\<e_-,.\>,\]
which is dual to the basis $(e_-, e_1, \ldots, e_n, \e_+)$, i.e.~$e^A(e_B)=\delta^A{}_B$.
We set
\[V^-:=\R\cdot \e^-=V_-^*,\qquad V^+:=\R\cdot \e^+=V_+^* ,\qquad V^0:=\Span(e^i)_{i=1,\ldots, n}=V_0^*.\]
The Lie algebra
\[\g:=\{X\in \gl(n+1,\R)\mid A^\top \I+\I A=0\},\]
is conjugated in  $\gl(n+1,\R)$ to $\so(1,n+1)$. Fixing the {\em grading element}
\[E:=\begin{pmatrix}
-1&0&0 \\0 &0& 0\\0&0&1
\end{pmatrix},\]
yields the splitting $V=V_-\+V_0\+V_+$ into the eigenspaces of $E$ in $V$, and 
  a $|1|$-grading of $\g$ into 
\[\g=\g_-\+\g_0\+\g_+,
\]
where $\g_\pm$  and $\g_0$  are the eigenspaces to the eigenvalue $\pm1$ and $0$ for $\mathrm{ad}_E=[E, -]$.
Explicitly,
\[
\g_0:=
\left\{(a,A):=
\begin{pmatrix}
a&0&0 \\0 &A& 0\\0&0&-a
\end{pmatrix}\mid a\in \R, A\in \so(n)\right\}= \R\cdot E\+\so(n)\simeq\co(n),\]
and
\[
\g_-:=
\left\{\overline{v}:=
\begin{pmatrix}
0&-v^\top&0 \\0 &0& v\\0&0&0
\end{pmatrix}\mid v\in V_0\right\},\qquad \g_+:=\g_-^\top.\]
We have the relations
\begin{equation}
\label{grading}\g_a(V_b) = V_{a+b},\quad\text{ 
where $a,b=-,0,+$ and $V_c=\{0\}$ if $c\not\in\{-,0,+\}$. }\end{equation}
Clearly, $\g_\pm$ are isomorphic to $V_0=\R^n$ and we define the projection
\[\begin{array}{crcl}
\pi_- : &\g &\longrightarrow & \g_-\ \simeq \ V_0\simeq \R^n
\\
&(a,A) +\overline{v} +\overline{w}^\top&\mapsto& \overline{v} \ \mapsto \ v.\end{array}\]
We slightly abuse notation by using $\pi_-$ for the map with range $\g_-$ but also with range $\R^n$. 

Similarly, we define $\pi_+( X^\top )= (\pi_- (X))^\top\in \g_+$. 
We also have $\g_0\simeq \co(n)$ and we define the projection
\[\begin{array}{crcccl}
\pi_0 : &\g &\longrightarrow & \g_0 &\simeq & \co(V_0)\simeq \R\+\so(V_0)\simeq \R\+\so(n)
\\
&(a,A) +\overline{v}+\overline{w}^\top&\longmapsto& (a,A) & \mapsto & a\Id+A.\end{array}\]
We also define the stabiliser of $\R\cdot e_-$,
\[\p:=\stab(\R\cdot \e_-)=\g_0\ltimes \g_-\simeq \co(V_0)\ltimes V_0=\co(n)\ltimes \R^n, \]
which is the Lie algebra of the Poincar\'{e} conformal group\footnote{We are aware that our convention for the grading element and hence of labelling $\g_\pm$ differs from the  standard reference \cite{cap-slovak-book1} on parabolic geometry, which is based on representation theoretic considerations.
%, but also to the convention in \cite{FinoLeistnerTaghavi-Chabert20}. 
We follow the convention in the literature relevant for Lorentzian homogeneous spaces, e.g.~in \cite{blau-oloughlin03},
%or \cite{Figueroa-OFarri00}, 
where the invariant null line is spanned by the first coordinate vector $e_-=\partial_-$, which forces us to denote $\p=\g_0\ltimes \g_-$ if we want to maintain the relation (\ref{grading}).}.
The Lie bracket in $\p$ is as follows
\[
\pi_0(
[X,Y])
=
\left[ \pi_0(X),\pi_0(Y)\right]_{\co(n)},\qquad  \pi_-(
[X,Y])= \pi_0(X) \cdot \pi_-(Y) - \pi_0(Y)\cdot  \pi_-(X),
\]
for $X,Y\in \p$ and where the dot denotes the representation of $\co(n)$ on $\R^n$.
Hence, $\g_-$ is an abelian ideal in $\p$. 
We have further projections
\[\pi_\R:\g_0\ni (a,A)\to a\in \R,\qquad \pi_{\so(n)}:\g_0\ni (a,A)\to A\in \so(n).\]

\subsection{Indecomposable subalgebras of $\g$}
We say that a subalgebra $\h$ of $\so(1,n+1)$ is {\em indecomposable} if there is no $\h$-invariant, non-degenerate subspace of Minkowski space. 
Clearly, irreducible subalgebras are indecomposable, and the first example of an indecomposable, non-irreducible algebra is $\so(1,1)$ which fixes two complementary null lines. Hence, from now on we assume that $n\ge 1$, so that $\so(1,n+1)$ is irreducible.
It is well-known \cite{benoist-harpe03,olmos-discala01,att05} that $\so(1,n+1)$ does not admit any proper irreducible subalgebras, so we will restrict to indecomposable, non-irreducible subalgebras. If $\h\subset \so(1,n+1)$ is an indecomposable, non-irreducible subalgebra, then it must be contained in the stabiliser of a null line, and hence is conjugated to a subalgebra of $\p\subset \g$, both  defined in the previous subsection.
If a subalgebra $\h$ in $\p$ is indecomposable, then $\pi_-(\h)=\g_-$. Moreover, $\h$ leaves invariant not just a null line but a null vector if and only if $\pi_\R(\h)=\{0\}$.
More importantly, indecomposable subalgrebas of $\p$ can be distinguished into four types.
\begin{theorem}[Berard-Bergery \& Ikemakhen \cite{bb-ike93}]\label{bb-ike-theo}
Let $\h$ be an indecomposable subalgebra of $\g\simeq\so(1,n+1)$ with $n\ge 1$. Let  $\h_0=\pi_0(\h)$ and let $\z$ and $\s$ be the centre and the semisimple part of  $\pi_{\so(n)}(\h_0)=\z\+\s$. Then $\pi_-(\h)=\g_-$. In addition, if $\g_-\subset \h$, then $\h=\h_0\ltimes \g_-$ and $\h$ is of 
\begin{description}
\item[type 1] $\R\subset \h$,
\item[type 2] $\pi_\R(\h)=0$, i.e. $ \h_0\subseteq \so(n)$, or
\item[type 3] there is a $\vf\in \z^*$ such that $\h_0=\mathrm{graph}(\vf)\+\s$.
\end{description}
Otherwise, i.e.~if $\g_-\not\subseteq \h$, then $\h_0\subseteq \so(n)$ and there is a splitting of $\g_-=V_0=V_1\+V_2$ and a surjective linear map $\psi:\z\to V_1 $ such that $\h_0(V_1)=0$ and  $\h=(\mathrm{graph}(\psi)\+\s)\ltimes V_2$. This is denoted as {\bf type 4}. Note that in this case we must have  $\dim(V_1)\ge 1$ and $\dim(V_2)\ge 2$ 

\end{theorem}

\begin{remark}
Theorem~\ref{bb-ike-theo} provides the foundation for the classification of indecomposable Lorentzian holonomy algebras. If such an algebra $\h$ is properly contained in $\g\simeq \so(1,n+1)$ it must be of one of the four types in the theorem. Moreover, it has been shown in \cite{leistnerjdg} that the $\so(n)$-projection of $\h$ is a Riemannian holonomy algebra, which led to a classification of possible holonomy algebras.
 For algebras of type $1$ and $2$ it is fairly straightforward to construct metrics with holonomy $\h$. For types $3$ and $4$ first examples of metrics appeared in  \cite{bb-ike93}, and in   \cite{galaev05} it was finally  shown that for every possible algebra $\h$ of type $3$ and $4$ with $\pi_{\so(n)}(\h)$ a Riemanian holonomy algebra there is a metric with holonomy $\h$. 
 \end{remark}

In the following we will require the following result about the centraliser in $\p$ of an indecomposable subalgebra of $\g$.
\begin{proposition}\label{centprop}
Let $\h$ be an indecomposable subalgebra of $\p$ and let $\z_\g(\h)$ be its centraliser in $\g\simeq \so(1,n+1)$ with $n\ge 1$. If $\pi_\R(\h)\not=\{0\}$,
then $
 \z_\g(\h)=\{0\}$, and otherwise \[\z_\g(\h) =\{\overline{z}\in\g_-\mid \pi_0(\h)\cdot z=0\}\subseteq \g_-.\]
 \end{proposition}
\begin{proof}
Let $Z\in \z_\g(\h)$. First  we note that if $Ze_-\not=0$, the line spanned by $Ze_-$ is left invariant by $\h$. Hence, since $\h$ is indecomposable, this line must be $\R\cdot e_-$, so that $Z\in \p$. With this, set $z:=\pi_-(Z)$ and $Z_0:=\pi_0(Z)$. For any $X\in \h$ we have 
\begin{equation}\label{cent}
0= \pi_-([X,Z]) =\pi_0(X)\cdot z - Z_0\cdot\pi_-(X).
\end{equation}
First assume that $\pi_\R(\h)\not=0$, i.e.~$\h$ is of type 1 or 3. Then we can choose $X\in \h$ such that $\pi_\R(X)\not=0$ and $\pi_-(X)=z$. Then from (\ref{cent}) we get
\[
0=(\pi_-([X,Z]))^\top z=(\pi_0(X-Z)z)^\top z =(\pi_\R(X)-\pi_\R(Z) ) z^\top z.
\]
Since for types 1 and 3 we can choose $X$ such that $\pi_\R(X)\not=\pi_\R(Z)$ independently of $\pi_-(X)$, we get $z=0$. With this, equation(\ref{cent}) implies that $0=Z_0\cdot\pi_-(X)$ for all $X\in \h$ and hence, with $\pi_-(\h)=V_0$, that $Z_0=0$. This implies  that $Z=0$ in the cases of Type 1 and 3.

If $\h$ is of type 2, we can choose $X\in \g_-\subset \h$ arbitrary to conclude from (\ref{cent}) that $Z_0=0$, so that $Z\in \g_-$.

Similarly, for type 4 we get that $Z_0$ acts trivially on $V_2$ and hence leaves $V_1$ invariant. From this  invariance and by recalling that  $\pi_0(X)$ acts trivially on $V_1$, we conclude that equation (\ref{cent}) for an arbitrary $X\in \mathrm{graph}(\vf)$ implies that $Z_0$ also acts trivially on $V_1$, so that $Z_0=0$.

Finally, with $Z=\overline{z}\in \g_-$, it is $[X,Z]=\pi_0(X)\cdot z$ and the statement follows.
\end{proof}

\subsection{Trivial submodules in $V^*\otimes \g$}

In this section, for a  indecomposable subalgebra in $\p$ of type 2 or 4,  we will determine the largest trivial submodule in $V^*\otimes \g$. 
We view $S\in V^*\otimes \g=\mathrm{Hom}(V,\g)$ as a homomorphism from $V\ni v$ to $S(v)\in\g$, so 
 that the representation of $X\in\h$ on $S\in V^*\otimes \g$ can conveniently be written as
\[ X\cdot S (v):=[X,S(v)] -S(Xv)\ \in \ \g,\]
for $v\in V=\R^{1,n+1}$. Here by $Xv$ we denote the action of $X\in \h$ on $v=a e_- +v_0+ be_+$, which is given as
\[Xv= \underbrace{(\pi_0(X)) \cdot v_0 +b \pi_-(X) }_{\in V_0} - (\pi_-(X)^\top v_0 )\,e_-.\]
A useful identity is also 
$\pi_-(X)=Xe_+$.
\begin{theorem}\label{algtheo}
Let $\h$ be an indecomposable subalgebra of $\p=(\R\+\so(n))\ltimes \R^n$ with $n\ge 2$,  and let $W$ be a trivial submodule of $\h$ in $V^*\otimes \g$.
If $\mathrm{pr}_\R(\h)\not=0$, then $W=\{0\}$. Otherwise, i.e.~if $\h\subset \so(n)\ltimes \R^n$, it holds that
\[W\subseteq (V^0\otimes \g_-) \+ (V^+\otimes \p),\] and every $S\in W$ satisfies
\begin{equation}\label{algtheoequ}S(x)e_+ =-\pi_0(S(e_+))\cdot x,  \text{ for all }x\in V_0,\quad \text{ and }\quad \left[ \h_0 , \pi_0(S(e_+)\right]=\{0\},\end{equation}
where $\h_0:=\pi_0(\h) \subseteq \so(n)$.
In addition, if $\h$ is of type 2, then $\h_0$ acts trivially on $ \pi_-(S(e_+))$, while in type 4 we have that
$\psi^{-1}(V_1)\+\s$ acts trivially on $\pi_-(S(e_+))$, where $\s$ is the semismple part of $\h_0$.
\end{theorem}
\begin{proof}
Let $W$ be a trivial submodule of $V^*\otimes \g$ for $\h$. Any $S\in W$ satisfies the equation
\begin{equation}
\label{trivial}
[X,S(v)] =S(Xv),\quad\text{ for all $v\in V=\R^{1,n+1}$ and all $X\in \h$.}\end{equation}
For $S\in W$, using the notations introduced above, we write
\begin{equation}\label{Ssplit}
S=\overline{s}_- +S_0+\overline{s}_+\in V^*\otimes \left(\g_-\+\g_0\+\g_+\right),\qquad s_\pm\in V^*\otimes V_0.\end{equation}
We consider the case $\mathrm{pr}_\R(\h)\not=0$ first, i.e.~the cases when $\h$ is either of type~1 or of type~3 in Proposition~\ref{bb-ike-theo}. In both cases $\g_-\subsetneq\h$, so that equation~(\ref{trivial}) implies that $[\g_-,S(e_-)]=0$, which yields $S(e_-)\in \g_-$. As a consequence,
$[\g_-,S(v_0)] \in \g_-$ for all $v_0\in V_0$, so that 
\[S(v_0)\in \p\quad\text{  for all $v_0\in V_0$.}\]
For {\bf type~1}, the grading element $E$ is in $\h$, so equation~(\ref{trivial}) for $X=E$ yields also that
\[S(e_+)\in \g_+,\qquad S(v)\in \g_0,\text{ for all} v\in V_0.\]
Since $n\ge 2$, for each $v\in V_0$ and $w\in v^\perp\subset V_0$ so that $X=\overline{w}\in \g_-\subset \h$  satisfies  $X(v)=0$ so that $0=[X,S(v)]=-S(v)\cdot w$. Hence, $S(v)\in \g_0\subset \so(n)$ acts trivially on the subspace $v^\perp$ and therefore on all of $V_0$, so that $S(v)=0$ for all $v\in V_0$. For all $X=\overline{v}\in \g_-$ this  implies 
\[0=S(Xe_+)=[X,S(e_+) ]=v^\top S(e_+), \qquad
0=[X,S(v)]=S(Xv)=v^\top vS(e_-),
\]
so that $S(e_+)=S(e_+)=0$. 

Now we consider {\bf type~3}, where there is a surjective linear map $\vf:\h_0\to \R$ such that $\h=\mathrm{graph}(\vf)\ltimes \g_-$. Let $X=\vf(A)E+A \in \h$ with $\vf(A)\not=0$. For $v=e_-$,  equation~(\ref{trivial}) gives
\[A \cdot s_-(e_-)=0,\]
whereas for 
$v=e_+$
it yields
\[ A\cdot s_-(e_+) =\vf(A) s_-(e_+),\quad  [A,S_0(e_+)]=\vf(A)S_0(e_+),\quad A\cdot s_+(e_+)=0.\]
Since $A\in \so(n)$, it has imaginary eigenvalues when acting on $V_0$ and as $[A,-]$ on $\co(n)$, so with  $\vf(A)\not=0$ the first two equations imply $ s_-(e_+) =0$ and $S_0(e_+)=0$.
Similarly,
for
$v_0\in V_0$ we get
\[ A\cdot s_-(v_0) =\vf(A) s_-(v_0),\quad  [A,S_0(v_0)]=S(Av_0).\]
where the  first equation implies $s_-(v_0)=0$.
So  we have that
\[S(e_-)=0,\qquad S(v_0)\in \g_0,\qquad S(e_+)\in \g_+,\]
for all $v_0\in V_0$.  With this, equation~(\ref{trivial}) yields that
$[\g_+,S(v_0)]=0$, so that $S(v_0)=0$,  and as a consequence  for all $X=\overline{w}\in \g_-$ that
\[S_0(e_+)\cdot w=-[X,S(e_+)]=-S(w)=0,\]
so that $S=0$.

Now we consider the case that 
$\pi_\R(\h)=\{0\}$, i.e.~that $\h\in \so(n)\ltimes \g_+$. Here we cannot use the grading element, but we have instead that $\h\cdot \e_-=\{0\}$. Hence, $S(e_-)$ commutes with $\h$ and therefore, by Proposition~\ref{centprop}, 
\[S(e_-)\in \{\overline{z}\in\g_-\mid \pi_0(\h)\cdot z=0\}\subseteq \g_-.\] 
Before we continue, we write equation~\ref{trivial} in its $\g_\pm$ and $\g_0$ components based on the Lie bracket in $\g$. For this we  split $X=X_0+\overline{x}\in\h\subseteq \h_0\ltimes V_0$ with $x\in V_0$, and $S$ as in ~(\ref{Ssplit}).
We already have that $s_+(e_-)=S_0(e_-)=0$.
Equation~(\ref{trivial}) is equivalent to the equations
\begin{eqnarray}
\label{trivial0}
0&=& [X_0,S_0(v)] + x s_+(v)^\top -s_+(v)x^\top +x^\top s_+(v)\Id - S_0(Xv),
\\
\label{trivial+}
0&=&-X_0\cdot s_+(v)-s_+(Xv),
\\
\label{trivial-}
0&=&X_0\cdot s_-(v) -S_0(v)\cdot x -s_-(Xv).
\end{eqnarray}
Now we consider these equations for $X=\overline{x}\in \g_-\cap \h$. Here, since  $Xe_i\in \R\cdot e_-$, so that $S_0(Xv)=0$, we get a pair of equations when $v=e_i$,
\begin{eqnarray}
\label{trivial0i}
0&=&  x s_+(e_i)^\top  -s_+(e_i)x^\top+x^\top s_+(e_i)\Id,
\\
\label{trivial-i}
0&=& -S_0(e_i)\cdot x +x^\top e_i s_-(e_-).
\end{eqnarray}
When $v=e_+$, noting that $Xe_+=x$, we get
\begin{eqnarray}
\label{trivial0+}
0&=& x s_+(e_+)^\top -s_+(e_i)x^\top+x^\top s_+(e_+)\Id - S_0(x),
\\
\label{trivial++}
0&=&s_+(x),
\\
\label{trivial-+}
0&=&-S_0(e_+)\cdot x -s_-(x).
\end{eqnarray}
Note that multiplying (\ref{trivial-i}) by $x^\top$ implies that 
\begin{equation}
\label{trivial-i2}
0=\pi_\R(S_0(e_i)) x^\top x - (x^\top e_i)(x^\top s_-(e_-)),\quad \text{ for all }x\in \h\cap \g_-.
\end{equation}
Since in both cases $\dim(\h\cap \g_-)\ge 2$,  for each $i=1, \ldots, n$ we can  find an $x\in \h\cap \g_-$ that is orthogonal to $e_i$. This implies that $S_0(e_i)\in \so(n)$ and as a consequence
\begin{equation}
\label{trivial-i2a}
0= (x^\top e_i)(x^\top s_-(e_-)),\quad \text{ for all }x\in \h\cap \g_-.
\end{equation}

From now on we consider the two remaining types separately. We start with {\bf type~2}, for which we have that $V_0\simeq \g_-\subset \h$, so we can choose $x\in V_0$ arbitrary. Then~(\ref{trivial++}) implies that $S(e_i)\in \p$ for all $i=1, \ldots ,n $ and ~(\ref{trivial-i2a}) implies that $S(e_-)=0$.
%
%
%In order to show that $S\in  (V^0\otimes \g_-) \+ (V^+\otimes \p)$ we work with 
% $X=\overline{x}\in \g_-$. For all $i=1, \ldots, n$ we have $Xe_i =-(x^\top e_i)e_-\in V_-$ so that $S(X e_i)\in \g_-$, so that equation~(\ref{trivial0}) for $v=e_i$ becomes 
%\begin{eqnarray}
%\label{xseie-}0 &=&
%xs_+(e_i)^\top-x^\top s_-(v)\Id \end{eqnarray}
%In type 2 we get this equation for  all $x\in V_0$, which shows that  $s_+(e_i)=0$, so that $S(e_i)\in \p$.
%Then equation~(\ref{trivial-}) becomes
%\begin{eqnarray}\label{trivial-i}
%0&=& -S_0(e_i)\cdot x + (x^\top e_i)\, S(e_-),\end{eqnarray}
%for all $i=1, \ldots, n$.
%In equation~(\ref{trivial-i}), for each $i$ we can choose $x=e_j$ with $j\not=i$ to get that  $S_0(e_i)\cdot e_j=0$, which implies that $S_0(e_i)  \in \so(n)$. Now multiplying~(\ref{trivial-i}) with $x^\top$  we get 
%$(x^\top e_i) (x^\top s_-(e_-))=0$ for all $x\in V_0$, which implies that 
%\[S(e_-)=0.\]
Then equation~(\ref{trivial-i}) also implies that $S_0(e_i)=0$, so that
$S(e_i)\in \g_-$  for all $i=1, \ldots, n$.
From this, equation~(\ref{trivial0+}) yields that 
 $S(e_+)\in \p$.
Then~(\ref{trivial-+}) implies that  $S_0(e_+)\cdot x+s_-(x)=0$, which is equivalent to the first equation in~(\ref{algtheoequ}).
 Finally, equation~(\ref{trivial0}) for 
 $X\in \h_0$ and $v=e_+$  yields the second equation in~(\ref{algtheoequ}), whereas~(\ref{trivial-}) shows that $\h_0$ acts trivially in $S_-(e_+)$. This finishes the proof for type~2.

Now we prove the theorem for an algebra $\h$ of {\bf type 4}, where we have that $\h\cap \g_-=V_2$.
Here we get equations~(\ref{trivial0i}--\ref{trivial-i2a}) only for $x\in V_2$. First we show that $S(e_-)=s_-(e_-)=0$. 
From~(\ref{trivial-i2a}) we have that $s_-(e_-)\in V_1$. 
Now we consider equation~(\ref{trivial-}) for  $X=X_0+y\in \mathrm{graph}(\psi)$, i.e. $X_0\in \z$ and $y=\psi(X_0)\in V_1$, and $v\in V_0$. Since $\h_0$ acts trivially on $V_1$, we have, $Xv=-y^\top v$ and equation~(\ref{trivial-}) becomes
\begin{equation}\label{trivial-a}
0=\underbrace{X_0\cdot s_-(v)}_{\in V_2} -S_0(v)\cdot y + (y^\top v) s_-(e_-).
\end{equation}
Since $\psi$ is surjective onto $V_1$ and $s_-(e_-)\in V_1$, we can choose $y=s_-(e_-)$, multiply this equation with $y$, and use $S_0(v)\in \so(n)$ to get
$(y^\top v)(y^\top y)=0$ for any $v\in V_0$. This implies $y=s_-(e_-)=0$ and so
$S(\e_-)=0$.
With this,~(\ref{trivial-i}) implies that $S_0(e_i)|_{V_2}=0$, so that $S_0(e_i)\in \so(V_1)$, but equation~(\ref{trivial-a}) then also shows that $S_0(v)=0$, for all $v\in V_0$.
Equations~(\ref{trivial0i}) and~(\ref{trivial++})  imply that
\[s_+(e_i)\in V_1,\text{ for all $i=1, \ldots, n$, and }\ s_+|_{V_2}=0.  \]
Equation~(\ref{trivial0}) for  $X=X_0+y\in \mathrm{graph}(\psi)$ together with $S_0(e_i)\in \so(n)$ then implies that $s_+(e_i)=0$ for all $i$, so that we have
$S(e_i)\in \g_-$. 
Equation~(\ref{trivial0+}), then implies that $s_+(e_+)=0$, so that
$S(e_+)\in \p$.

Finally, equation~(\ref{trivial-+}) is the first equation in~(\ref{algtheoequ}) and~(\ref{trivial0}) the second, and~(\ref{trivial-}) shows that $\psi^{-1}(V_1)\+\s$ acts trivially on $\pi_-(S(e_+))$. This completese the proof for type~4, and hence the proof of the theorem.
\end{proof}
Writing out the statement of Theorem~\ref{algtheo} more explicitly, we have that any $S\in V^*\otimes \g$ such that $\h\cdot S =0 $ is of the form
\begin{equation}\label{Sexplicit}
S(e_+)=\begin{pmatrix}
a &  -v^\top & 0 \\
0 &A & v
\\0&0&-a
\end{pmatrix},
\quad
S(x)
=\begin{pmatrix}
0 &  x^\top(a-A) & 0 \\
0 &0 & -(A+a)x
\\0&0&-a
\end{pmatrix},
\text{ for  }x\in V_0,\end{equation}
and $ S(e_-)=0$.
We  remark that 
 the condition $n\ge 2$ is sharp, as we will see in Section~\ref{dim3sec}.
%\begin{example}
%\label{example1}
%Consider $3$-dimensional Minkowski space $V$ with Witt basis $(e_-,e_1,e_+)$ and $\h=\R\cdot X=\R\cdot \overline{e_1}$, i.e.
%\[X=\begin{pmatrix}
%0&-1&0\\ 0&0&1 \\0&0&0\end{pmatrix}.\]
%Define $S\in V^*\otimes \g$ by 
%\[S(e_-)=a X,
%\quad 
%S(e_1)= 
%\begin{pmatrix}
%a&-b&0\\ 0&0&b \\0&0&-a\end{pmatrix} \not\in \g_-,
%\quad
%S(e_+)= 
%\begin{pmatrix}
%-b&-c&0\\ -a&0&c \\0&a&b\end{pmatrix}\not\in \p,
%\]
%for real numbers $a$, $b$ and $c$. This means that $S\not\in ( V^0\otimes \g_-)\+ V^+\otimes \p$. Nevertheless, it is $X\cdot S=0$. Indeed, 
%\[ [X,S(e_-)]=0,\quad [X,S(e_1)] =-aX,\quad   [X,S(e_+)] =S(e_1),\]
%so that 
%\[ X\cdot S (e_1)=[X, S(e_1)] -S(Xe_1)= -aX +S(e_-)=0,\]
%and 
%\[ X\cdot S (e_+)=[X, S(e_+)] -S(Xe_+)= S(e_1) -S(e_1)=0.\]
%Moreover, it is easy to check that the above $S$ is the most general element  in the submodule of  $V^*\otimes \g$ on which $\h$ acts trivially. This shows that this submodule is of dimension $3$, but contain an $S$ that does not satisfy the conclusions of Theorem~\ref{algtheo}.
%\end{example}

\subsection{Trivial submodules in the torsion module}
From the result in the previous section we can easily determine the trivial submodules in the torsion module $\Lambda^2V^*\otimes V$ using the $\g$-equivariant isomorphism
between 
\begin{equation}\label{iso}V^*\otimes \g  \simeq \Lambda^2V^*\otimes V,\end{equation}
that assigns to $S\in V^*\otimes \g$ its skew symmetrisation in the first two components,
\[ T(v,w):=\tfrac{1}{2}\left( S(v)w-S(w)v\right).\]
From Theorem~\ref{algtheo} we obtain the following consequence.
\begin{corollary}\label{Talgcor}

Let $\h$ be an indecomposable subalgebra of $\p=(\R\+\so(n))\ltimes \R^n$ with $n\ge 2$,  and let $W$ be a trivial submodule of $\h$ in $V^*\otimes \g$.
If $\mathrm{pr}_\R(\h)\not=0$, then $W=\{0\}$. Otherwise, i.e.~if $\h\subset \so(n)\ltimes \R^n$, it holds that
\[W\subseteq  \left(( V\wedge V^+) \+ (V^0\wedge V^0 )\right) \otimes V_-\  \+\ (V^0\wedge V^+) \otimes V_0.\] Moreover, if for  $T\in W$, 
we define $b\in \R$ and $\omega\in \Lambda^2V_0$ by
\[ T(e_+,\e_-)=b\,e_-,\qquad \omega(x,y):=\< T(x,y),\e_+\>,\]
then 
\begin{equation}\label{Talgcorequ}
\< T(e_+,x),y\>=b\, x^\top y +\omega (x,y).\end{equation}
\end{corollary}
\begin{proof}
Since the isomorphism~(\ref{iso}) is $\g$-equivariant, every $T\in W$ is the skew symmetrisation of an $S\in V^*\otimes \g$ such that $\h\cdot S=0$. By Theorem~\ref{algtheo}, if $\mathrm{pr}_\R(\h)\not=0$, then $S=0$ and otherwise $S$ is in $(V^0\otimes \g_-)\+(V^+\otimes \p)$  and of the form~(\ref{Sexplicit}), with the parameters $a\in \R$, $A\in so(n)$ and $v\in V_0$. For the skew symmetrisation this yields
\[T(e_+,e_-) =\tfrac{1}{2}S(e_+)e_-=\tfrac{a}{2} e_-,\]
and
\[\< T(x,y),e_+\> = \tfrac{1}{2}\<  S(x)y-S(y)x,\e_+\> = \tfrac{1}{2}\left( x^\top(a-A)y - y^\top(a-A)x\right) =-x^\top A  y,\]
as well as
\[T(e_+,x)= \tfrac{1}{2}\left( Ax -v^\top x e_- +(A+a)x\right)=Ax+\tfrac{1}{2}(ax  -v^\top x e_-). \]
This shows $\< T(e_+,x),y\> = (Ax)^\top y + \tfrac{a }{2}x^\top y $, which proves the statement.
\end{proof}
\begin{remark}
In regards to the splitting of the torsion modules into irreducible $\g$-modules with vectorial, twistorial or totally skew torsion tensors, as explained in the introduction, one easily obtains the following direct characterisation in terms of $S$. If $n\ge 2$ and $T$ is the skew symmetrisation of $S$ with  $\h\cdot T=0$, then 
\begin{enumerate}
\item $T$ is totally skew $\iff$ $S(e_+)\in \so(n)$  $\iff$ $T(e_+)\in \so(n)$ 
\item $T$ is  twistorial $\iff$ $S(e_+)\in \g_-$,  $\iff$ $T(e_+)\in V^0\otimes V_-$, and
\item $T$ is vectorial $\iff$ $S(e_+)\in \R$ $\iff$ $T(e_+)\in V^-\otimes V_-$. 
\end{enumerate}
The details can be found in \cite{steven-thesis}.
\end{remark}

\subsection{Algebraic curvature tensors}
\label{curvsec}
Motivated by the results in later sections, here we will study the algebraic curvature tensors that come from connections with torsion. We will follow the approach in \cite{galaev03} but we will allow for a very special torsion.

Let $(V,\<.,.\>)$ be a semi-Euclidean vector space,  $T\in \Lambda^2V^* \otimes V$, and let $\h$ be a subalgebra of $\so(V)$. In this setting we define
\[
\mathcal R(V,\h,T):= \{ R\in \Lambda^2V^*\otimes \h\mid \underset{u,v,w}{\mathfrak{S}}    \left( R(u,v)w +T(T(u,v),w)\right)=0\text{ for all }u,v,w\in V\},\]
where $\underset{u,v,w}{\mathfrak{S}}   $ denotes the sum of the cyclic permutations of $u,v,w$, e.g.\[\underset{u,v,w}{\mathfrak{S}} R(u,v)w= R(u,v)w+ R(v,w)u+ R(w,u)v.\]
If $T=0$, we denote $\mathcal R(V,\h,T)$ by $\mathcal R(V,\h)$.
Since $ \mathcal R (V,\h,T)\subset \Lambda^2V^*\otimes \so(V)$, every $R\in  \mathcal R (V,\h,T)$ satisfies
\[\< R(u,v)w,z\>+\<R(u,v)z,w\>=0, \] and 
consequently
satisfies the following pairwise symmetry, 
\begin{eqnarray} 
\lefteqn{2(\< R(u,v)w,z\>-\< R(w,z)u,v\>)=\nonumber}
\\
&=&\label{pair}
\< \underset{u,v,w}{\mathfrak{S}}   T(T(u,v),w),z\>+\< \underset{w,z,v}{\mathfrak{S}} T(T(w,z),v),u\> 
\\
&&\nonumber
+\< \underset{z,w,u}{\mathfrak{S}} T(T(z,w),u),v\>+\< \underset{v,u,z}{\mathfrak{S}} T(T(v,u),z),w\>.
\end{eqnarray}

As in the previous subsections, let 
$V=\R^{1,n+1}$
and in the following $\h$ is
an indecomposable subalgebra  of $\so(n)\ltimes \g_-\subset \p$, i.e.~ a subalgebra of $\p$ with $\pi_\R(\h)=\{0\}$. We continue to use the same notation as in the previous subsections, $V=V_-\+V_0\+V_+$, etc. In this setting we prove the following.
\begin{theorem}\label{curvtheo}
Let  $\h$ an indecomposable subalgebra of $\so(n)\ltimes \g_-$, with $n\ge 1$, and $\h_0=\pi_{\so(n)}(\h)$. If $T\in \Lambda^2 V^*\otimes V$ 
satisfies
\begin{equation}
\label{Tcond}
T\ \in \left( ( V\wedge V^+) \+ \Lambda^2 V^0\right) \otimes V_-\  \+\ (V^0\wedge V^+) \otimes V_0,\end{equation}
then  $\mathcal R(V,\h,T)$ injects into 
\[
\mathcal R(V_0,\h_0)\+
\mathcal P(V_0,\h_0) \+ 
V_0\otimes V_0,\]
where \[ \mathcal P(V_0, \h_0):=\{P\in V^0\otimes \h_0\mid  \underset{x,y,z}{\mathfrak{S}}     \left(P(x,y)\right)^\top z=0\text{ for all }x,y,z\in V_0
\}.\] 
\end{theorem}
It is remarkable that none of the spaces in the theorem depend on $T$ even though $\mathcal R(V,\h,T)$ does. The proof will proceed by several auxiliary statements that are useful in their own right.
\begin{proof}
The assumption (\ref{Tcond})
 implies that 
\begin{equation}\label{e-T}
T(T(u,v),\e_-)=0,\qquad \<T(u,v),e_-\>=0,\quad \text{ for all }u,v\in V,
\end{equation}
and 
\begin{equation}\label{e-T2}
T(T(e_-,u),x)=0,\qquad \<T(T(e_-,u),v),x\>=0,\quad \text{ for all }u,v\in V, x\in V_0
\end{equation}

as well as 
\begin{equation}\label{eiT}
T(T(x,y),z)= 0\qquad\< T(x,y),z\>=0,\quad \text{ for all }x,y,z\in V_0.
\end{equation}

From this we deduce the following.
\begin{lemma}\label{symlemma}
Let $T$ be as in (\ref{Tcond}), $\h$ an indecoposable subalgebra of $\so(n)\ltimes \g_-$, and $R\in\mathcal R(V,\h,T)$.
Then 
\begin{enumerate}
\item $\underset{x,y,z}{\mathfrak{S}} R(x,y)z=0$ for all $x,y,z\in V_0$,
\item $\< R(u,v)w,z\>-\< R(w,z)u,v\>=0$ if three of $u,v,w,z$ are in $V_-\+V_0$, and
\item  $R(e_-,v)=0$ for all $v\in V$.
\end{enumerate}
\end{lemma}
\begin{proof} 
For the first statement, consider equation (\ref{pair}) with $u,v,w\in V_-\+V_0$. With (\ref{Tcond}) and consequently (\ref{eiT}),the right-hand-side of  (\ref{pair}) vanishes.

With $\h\in \so(n)\ltimes \g_-$ which is the annihilator of $e_-$ in $\so(1,n+1)$ we clearly have $R(v,w)e_-=0$. With the first statement, we only have to check that $ 
\<R(e_-,e_+)x,e_+\>=0$ for all $x\in V_0$. However this holds because of   (\ref{pair}), (\ref{e-T}) and (\ref{e-T2}).
%(\ref{e-T}) we have
%\begin{eqnarray*}
%2\<R(e_-,e_+)x,e_+\> &=& \<T(T(e_-,v),w) - T(T(e_-,w),v),z\>
%\\
%&&
%+\<T(T(e_-,z),w) -T(T(e_-,w),z),v\>
%\\
%&&
%+\<T(T(e_-,z),v) -T(T(e_-,v),z),w\>
%\end{eqnarray*}
%for all $v,w,z\in V$. With the assumption (\ref{Tcond}), the right hand side is only zero if all $v,w,z\in V_+$. Since $R\in \Lambda^2V^*\otimes \so(1,n+1)$, this however implies that $R(e_-,v)=0$ for all $v\in V$.
\end{proof}

With this observation and using the extra structure of $V=V_-\+V_0\+V_+$, every $R\in \mathcal R (V,\h,T)$ is an element in 
\begin{eqnarray*}
 \left(\Lambda^2V^0 \+ (V^0\wedge V^+)  \right) \otimes  \h
&=&
\Lambda^2V^0\otimes\h\ \+\  (V^+\wedge V^0)\otimes \h.
\end{eqnarray*}
%We define
%\[\mathcal R_0(\h)=\{ Q\in \Lambda^2V^0\otimes\h \mid \underset{x,y,z}{\mathfrak{S}}     Q(x,y)z=0\}\]
We will now use the inclusion $\h\subseteq \h_0\ltimes \g_-$, with $\h_0=\pi_{\so(n)}(\h)$,  to describe the vector space $\mathcal R (V,\h,T)$. For this we consider the following $2$-form $\omega_T\in \Lambda^2 V^0$ defined by $T$,
\begin{equation}\label{omegaT}
\omega_T(x,y):= 
\Big\<e_+, T(T(x,y),e_+)+T(T(e_+,x),y)-T(T(e_+,y),x)\Big\>.\end{equation}
\begin{lemma}\label{lemma1}
Let $T$ be as in (\ref{Tcond}), $\omega_T$ defined by $T$ as in~(\ref{omegaT}), and $\h$ an indecomposable subalgebra of $\so(n)\ltimes \g_-$. Then $\mathcal R(V,\h,T)$ injects into 
\[\mathcal Q (V_0,\h,\omega_T)\+\mathcal R(V_0,\h_0),\]
where 
\[\mathcal Q (V_0,\h,\omega_T) = \left\{Q\in V^0\otimes \h\mid
\begin{array}{l}
\underset{x,y,z}{\mathfrak{S}}    \< Q(x,y),z\>=0,\\
\big\<e_+,  \underset{x,y}\Wedge Q(x,y)\big\>=\omega_T(x,y),\end{array}\text{ for all }x,y,z\in V_0\right\},
 \]
 where $\underset{x,y}\Wedge Q(x,y):= Q(x,y)- Q(y,x)$.
 Moreover, if $\h$ is of type 2, then this injection is an isomorphism.
 \end{lemma}
\begin{proof}
The injection is defined by mapping an $R
\in \mathcal R(V,\h,T)$ to the pair $(Q,R_0)$ where
\begin{equation}\label{pq}Q(x)=R(e_+,x)\in \h,\qquad R_0(x,y)=\pi_0(R(x,y)),\qquad \text{for }x,y\in V_0.\end{equation}
Because of the first point in Lemma~\ref{symlemma}, $R_0$ satisfies the defining condition of $\mathcal R(V_0,\h_0)$. Similarly, again from Lemma~\ref{symlemma}, we have for all $x,y,z\in V_0$ that
\[
\underset{x,y,z}{\mathfrak{S}}    \< Q(x,y),z\>= 
\underset{x,y,z}{\mathfrak{S}}    \< R(e_+,x) y,z\>
=
-\underset{x,y,z}{\mathfrak{S}}    \< R(y,z) x,\e_+\>
=0,
\]
so $Q$ satisfies the first condition for $\mathcal Q(V_0,\h,\omega_T)$. 
To check the second condition we use the defining condition for $\mathcal R(V,\h,T)$ to get
\begin{eqnarray*}
Q(x,y)-Q(y,x)
&=&R(e_+,x)y-R(e_+,y)x
\\
&=&
R(y,x)e_+ -T(T(x,y),e_+) -T (T(e_+,x),y)+T (T(e_+,y),x).
\end{eqnarray*}
Multiplying by $\<e_+,.\>$ and recalling that $\h\in \g$, so that $\<R(x,y)e_+,e_+\>=0$, yields the second equation for $Q$.

The injectivity of the map follows from $e_-\hook R=0$, and from Lemma~\ref{symlemma} with
\[\<R(x,y)e_+,z\>=\<R(e_+,z)x,y\>=\<Q(z,x),y\>.\]

If $\h$ is of type 2,~i.e. $\h= \h_0\ltimes \g_-=\h_0\ltimes V_0$, then we can define an inverse as follows.
Let $Q\in\mathcal Q (V_0,\h,\omega_T)  $ and $R_0\in \mathcal R(V_0,\h_0)$, and define $R\in \Lambda^2V^* \otimes \h $ by $e_-\hook R=0$ and (\ref{pq})~i.e. 
by \[\pi_0(R(x,y))=R_0(x,y),\quad R(e_+,x)=Q(x),\]
and 
\[ \<R(x,y)e_+,z\>= \<Q(z,x),y\>,\]
for all $x,y,z\in V_0$.
Then  $R\in  \mathcal R (V,\h,T)$.
\end{proof}

%
%{\red
%
%\begin{lemma}
%Let  $\h$ an indecomposable subalgebra of $\so(n)\ltimes \g_-$ and $\h_0=\pi_{\so(n)}(\h)$. Then 
%$\mathcal R_0(\h)$ injects into 
%\[\mathcal R(V_0,\h_0)\+ \mathcal P( V_0),\]
%where \[ \mathcal P( V_0):=\{P\in \Wedge^2V^0 \otimes V_0\mid \underset{x,y,z}{\mathfrak{S}}     P(x,y)^\top z=0\}.\] Moreover, if $\h$ is of type 2, then this injection is an isomorphism.
%\end{lemma}
%\begin{proof}
%The injection is given by  \[ \mathcal R_0(\h)\ni R \mapsto \left( \pi_0\circ R, P\right),\]
%where $P(x,y):= R(x,y) e_+$. 
%It is clear from the symmetries of $R$ that $ \pi_0\circ R \in\mathcal R(V_0,\h_0) $ and $P \in \mathcal P( V_0)$. 
%Moreover, if $\h$ is of type $2$, it is $\h=\h_0\ltimes \g_-\simeq \h_0\ltimes V^*0$, so assigning to the pair $(R_0,P)$ the element $R\in \mathcal R_0(\h)$ by
%\[R(x,y):= (R_0(x,y), P(x,y))\in \h_0\ltimes V_0\]
%shows surjectivity.
%\end{proof}}

\begin{lemma}\label{lemma2}
Let $T$ be as in (\ref{Tcond}) and $\h$ an indecomposable subalgebra of $\so(n)\ltimes \g_-$ and $\h_0=\pi_{\so(n)}(\h)$. 
 Then 
$\mathcal Q(V_0,\h,\omega_T)$ injects into 
\[\mathcal P(V_0,\h_0) \+ 
%\mathcal L (V_0,T), {\red \ \ \ or \  
%\mathcal P(V,\h_0)\+
 (V_0\otimes V_0)\]
where
$\mathcal P(V_0,\h_0)$ is defined in Theorem~\ref{curvtheo}. 
%and 
%\[\mathcal L (V_0,T):=\{ L \in V^0\otimes V^0\mid  L(x,y) = \underset{x,y\in V_0}\Wedge \left( L(x,y)+\<T(T(e_+,x),y),e_+\>\right)=0\}.\]
%Moreover, if $\h$ is of type 2, then this injection is an isomorphism.
\end{lemma}
\begin{proof}
The injection assigns to $Q\in \mathcal Q(V_0,\h,\omega_T)$ the pair
$(P , L )$, where 
\[P(x):=\pi_0( Q(x))\in \h_0,\qquad L(x,y):=\< Q(x)e_+,y\>=\<\pi_-(Q(x)),y\>.\] 
By the definition of $\mathcal Q(V_0,\h,\omega_T)$, $P\in V^0\otimes \h_0$ has the correct symmetries to be  in $\mathcal P(V_0,\h_0)$. 
The injectivity follows immediately for $Q(x)\in \h\subseteq\h_0\ltimes V_0\subset \g$. 
\end{proof}
Composing the injections in Lemmas~\ref{lemma1}~and~\ref{lemma2} concludes the proof of Theorem~\ref{curvtheo}.
\end{proof}

\begin{remark}
In Lemma~\ref{lemma2}, when $\h$ is of type 2,  it is not hard to see that the range is given by the direct sum of 
$\mathcal P(V_0,\h_0)\+ \mathcal L_{\omega_T}$, where
\[\mathcal L_{\omega} := \{L\in V^0\otimes V^0\mid \underset{x,y}\Wedge L(x,y) =\omega(x,y)\},\]
to get the isomorphism 
\[\mathcal R(V,\h,T)\simeq
\mathcal R(V_0,\h_0)\+
\mathcal P(V_0,\h_0) \+ \mathcal L_{\omega_T}
\]
 in Theorem~\ref{curvtheo}.
 When  $\h$ is of type~4 with $\h\cap \g_-=V_2$ and defined by $\psi:\h_0\to V_1$ as in Theorem~\ref{bb-ike-theo},
 analysing the range in  in Lemmas~\ref{lemma1}~and~\ref{lemma2}  yields to the isomorphism
\[\mathcal R(V,\h,T)\simeq \big\{ (R,P,L)\in 
\mathcal R(V_0,\ker(\psi))\+
\mathcal P(V_2,\h_0) \+ \mathcal L_{\omega_T}\mid \mathrm{pr}_{V_1} \circ L=\psi\circ P\big\}
\]
 in Theorem~\ref{curvtheo}. 
Details are provided in \cite{steven-thesis}.
\end{remark}
For $T$ as in (\ref{Tcond}) and an indecomposable subalgebra of $\so(V_0)\ltimes \g_-$,   we  will now use Theorem~\ref{curvtheo} to describe 
 trivial submodules in $\mathcal R(V,\h,T)$. 
\begin{corollary}\label{curvtheo2}
Let  $\h$ an indecomposable subalgebra of $\so(n)\ltimes \g_-$, with $n\ge 1$  and $\h_0=\pi_{\so(n)}(\h)$. Assume that $T\in \Lambda^2 V^*\otimes V$ satisfies 
condition (\ref{Tcond}) and that $R\in \mathcal R(V,\h,T)$. 
If $\h\cdot R=0$, then $R(u,v)\in \g_-$ for all $u,v\in V$.
\end{corollary}
\begin{proof}
Let $R\in \mathcal R(V,\h,T)$ and define $R_0\in \mathcal R (V_0,\h_0)$ and $P\in \mathcal P(V_0,\h_0)$, as
\[
R_0(x,y)=\pi_0(R(x,y)),\qquad P(x)=\pi_0(R(e_+,x)),\qquad\text{for }x,y\in V_0.\]
This implies that 
\[
R(x,y)z=R_0(x,y)z - \< P(z)x,y\> e_-,\]
and hence
\[\pi_-(R(x,y))=P(x)y-P(y)x.
\]
From Theorem~\ref{curvtheo} it follows that $R(u,v)\in \g_-$ if and only  if $R_0=0$ and $P=0$. 

For a subalgebra $\h$ of type 2, the proof is relatively straightforward. In this case we take any  $Z_-\in \g_-\subseteq \h$ to obtain for all $x,y\in V_0$ that
\begin{equation}\label{Z+R}
0
=
(Z_-\cdot R)(x,y) = \left[ Z_-,R(x,y)\right] -R(Z_- x,y) - R(x,Z_-y) =\left[ Z_-,R(x,y)\right] , \end{equation}
because  $Z_-x\in V_-$ and $e_-\hook R=0$. This means that  $R(x,y)$ is contained in the centraliser in $\h$ of $\g_-$ and hence, by Proposition~\ref{centprop}, contained in $\g_-$. This implies that 
$R_0=0$. Then for the $\g_-$-component of $(Z_-\cdot R)(e_+,y)$ we get
for all  $y\in V_0$ that 
\begin{equation}\label{Z+R2}
0
 = \pi_-(\left[ Z_-,R(e_+,y)\right]) -\pi_-(R(Z_- e_+,y))
 =
 P(y)Z_-
 -
 P(Z_-)y-P(y)Z_-=-
 P(Z_-)y
  .\end{equation}
 Since $Z_-$ was arbitrary in $\g_-\simeq V_0$, this implies that $P=0$.

For $\h$ of type 4, we have $\h_0\subseteq \so(V_2)$ so that $R_0(x,y)$ and $P(x)$ are also  in $\so(V_2)$. Now we have equation~(\ref{Z+R}) only for $Z_-\in \H\cap \g_-=V_2$, so that $R(x,y)$ is in the centraliser in $\h$ of $V_2$, which is contained in $\so(V_1)\ltimes \g_-$. But with  $\h_0\subseteq \so(V_2)$, this implies that $R(x,y) \in \g_-\cap \h=V_2$ for all $x,y\in V_0$. Moreover, for type 4 we have equation (\ref{Z+R2}) only for $Z_-\in V_2$, so that $P(z)=0$ for all $z\in V_2$. Hence,  $P\in V^1\otimes \so(V_2)$, and the symmetry of $P$ implies that $P=0$.
\end{proof}

\section{Lorentzian Ambrose--Singer connections}
\label{LASCsec}
\subsection{Lorentzian connections with special holonomy and their screen bundle} 
\label{screensec}
Let $(M,g)$ be a Lorentzian manifold. We call a connection $\tnab$ on $TM$ a {\em Lorentzian connection} if $\tnab g=0$. We denote the curvature of $\tnab$ by $\tR$ and its holonomy algebra by $\widetilde{\hol}$.
If the holonomy group is indecomposable, non-irreducible, $TM$ admits a parallel null line subbundle  $\mathcal X\subset TM$. 
Since $\tnab g=0$, its orthogonal space $\mathcal X^\perp$, which contains $\mathcal X$ is also parallel, so the tangent bundle is filtered by parallel subbundles of rank and co-rank $1$ as
\[ \mathcal X\subset \mathcal X^\perp\subset  TM.\]
%We define the following vector bundles
%\[\G:=\so(TM,g)=\{A\in \End(TM)\mid g(AX,Y)+g(X,AY)=0\},\]
%as well as 
%\[\P:=\{A\in \G\mid A\mathcal X\subseteq \mathcal X\}\qquad\text{
%and }\qquad
%\G_-:=\{A\in \P\mid A(\mathcal X)=0, A(\mathcal X^\perp)\subseteq \mathcal X\}.\]
%Their fibres correspond to the Lie algebras $\g=\so(1,n+1)$, $\p=\stab_\g(\R\cdot \e_-)$ and $\g_-$ defined in the previous section.
The parallel filtration of $TM$ allows to define the screen bundle 
\[\E:=\mathcal X^\perp/\mathcal X.\] 
Since $\mathcal X$ is null and parallel,   the metric and the connection induce a positive definite metric $g^{\E}$ and a connection $\tnab^{\E}$ on $\E$ via
\[ g^{\E}([X],[Y]):= g(X,Y),\qquad \tnabE_V[\eta]:=[\tnab_V \eta ],\]
for all $V\in TM$, $X,Y\in \mathcal X^\perp$ and $\eta$ a section of $\mathcal X^\perp$, and where $[.]$ denotes the equivalence class in $\E$.
This immediately implies for the curvature of $\tnabE$ that
\[\tRE(U,V)X=[ \tR (U,V) X],\]
for all $U,V\in TM$ and $X\in \mathcal X^\perp$.
Therefore, $\tnab^{\E}$ is flat if and only if  \[g(\tR (U,V) X,Y)=0\text{ for all $U,V\in TM$ and $X,Y\in \mathcal X^\perp$.}\]
Moreover, in this context, we have the following statement, which was proven for the Levi-Civita connection in \cite{leistner05c}.
\begin{proposition}
Let $\gamma$ be a curve from $p$ to $q$ in $M$, $P_\gamma$ and $P^{\E}_\gamma$ be the parallel transports with respect to the connections $\tnab$ and $\tnabE$. Then the following diagram commutes
\begin{eqnarray}\nonumber
\mathcal X^\perp|_p& \stackrel{P_\gamma}{\longrightarrow} & \mathcal X^\perp|_q
\\
\label{cd}
\downarrow&&\downarrow
\\
\nonumber
\E|_p& \stackrel{P^\E_\gamma}{\longrightarrow} & \E|_q.\end{eqnarray}
In particular, the holonomy group and algebra of $\tnabE$ satisfy 
\[
\Hol (\E,\tnab) =\pi_{\O(n)}( \Hol(TM,\tnab), \qquad
\hol(\E,\tnabE)=\pi_{\so(n)}( \hol(TM,\tnab),\]
where $\pi_{\so(n)}: \g\to \so(n)$ is the projection defined in the previous section and $\pi_{\O(n)}$ the corresponding projection of the Lie groups.
\end{proposition}
\begin{proof}
Let $\gamma:[0,1]\to M$ be a curve in $M$ with $\gamma(0)=p$ and $\gamma(1)=q$.
From the definition of $\tnab^\E$ it follows that
\[
\frac{\tnabE [X]}{dt}=\bigg[\frac{\tnab X }{dt}\bigg],\]
for $X:[0,1]\to \mathcal X^\perp$ a section of $\gamma^*\mathcal X^\perp$, where the brackets denote the equivalence classes.
Hence, for the diagram~(\ref{cd}) to commute, we have to show  for two sections $X$ and $Y$ of $\gamma^*\mathcal X^\perp$ such that $X$ is $\tnab$-parallel transported and $Y$ satisfies 
\begin{equation}\label{nabYt} \frac{\tnab Y}{\d t}\in \Gamma(\gamma^*\mathcal X),\qquad 
Y(0)-X(0)\in \mathcal X|_p,
\end{equation}
it holds that $X-Y$ is a section of $\gamma^*\mathcal X^\perp$.
For this, let $e_-, e_1, \ldots, e_m,e_+$ be a Witt basis of $T_pM$ with $e_-\in \mathcal X|_p$. Let $E_-\in \Gamma (\gamma^*\mathcal X )$ and $E_i \in \Gamma (\gamma^*\mathcal X^\perp )$, $i=1, \ldots, n$  be the parallel transported frame of $\gamma^*\mathcal X^\perp$, obtained from $(e_-, e_1,\ldots, e_n)$. 
Then, for any  section $Y$ of $\gamma^*\mathcal X^\perp$ we have
 \[Y(t)=Y^-(t) E_-(t)+Y^i(t)E_i(t),\qquad \frac{\tnab Y}{dt}= (Y^-)^\prime E_-+(Y^i)^\prime E_i,
 \]
 for some functions $Y^A$ of $t$.
If $X$ is parallel transported, then
$X(t)=a^-E_-(t)+ a^i E_i(t)$, 
for constants $a^-$ and $a^i$. If $Y$ satisfies~(\ref{nabYt}), then  $Y^i\equiv a^i$ constant, for $i=1, \ldots, n$.
Hence, $X(t)-Y(t)\in\mathcal X|_{\gamma(t)}$. This concludes the proof. 
\end{proof}

\subsection{Lorentzian $2$-symmetric connections}
We say that a Lorentzian connection~$\tnab$ with curvature $\tR$ is {\em second order symmetric} or {\em $2$-symmetric}, if $\tnab\tnab\tR=0$. 
Of course, this includes {\em locally symmetric Lorentzian connections}, i.e.~connections with $\tnab\tR=0$, and the Levi-Civita connections of second order symmetric spaces \cite{senovilla08,AlekseevskyGalaev11} and of locally symmetric spaces.
In this subsection we will prove the following result about locally  Lorentzian $2$-symmetric connections.
\begin{theorem}\label{VFtheo}
Let $\tnab$ be a  Lorentzian  $2$-symmetric connection on $(M,g)$ with $M$  of dimension $m\ge 4$ with indecomposable, non-irreducible holonomy algebra. Then, on the universal cover of $M$, $\tnab$  admits a parallel null vector field. In particular, the holonomy algebra of $\tnab$ is contained in $\so(m-2)\ltimes \R^{m-2}$. For locally symmetric connections, the result also holds for $m=3$.
\end{theorem}
The proof of this theorem is based an algebraic result from the previous section and on two lemmas, the first of which is well-known for the Levi-Civita connection. 
\begin{lemma}\label{VFlem1}
Let $M$ be a simply connected manifold, $\tnab$ an affine connection, and $\xi$  a $\tnab$-recurrent vector field,~i.e. a vector field with $\tnab\xi=\theta\otimes \xi$, for $\theta$ a section of $T^*M$. Then $\xi$ can be rescaled to a parallel vector field if and only if $\d \theta=0$, or equivalently, if $\tR(X,Y)\xi=0$.
\end{lemma}
\begin{proof}
The recurrent vector field $\xi$ can be rescaled to a parallel vector field if and only if there is a  smooth function $h$ such that $\hat\xi=h\xi$ is parallel, i.e. if and only if
\[ \d h + h\theta =0.\]
Clearly, if such a function exists, $\theta$ is closed. On the other hand, if $\theta$ is closed, since $M$ is simply connected, there is a smooth function $f$ such that $\theta=\d f$. Then the function $h=\e^{-f}$ will solve the above equation. 

The other equivalence follows from the identity
\begin{equation}
\label{recurv}
\tR(X,Y)\xi=\d\theta (X,Y)\xi,\end{equation}
that is satisfied by every recurrent vector field.
\end{proof}
\begin{lemma}\label{VFlem3}
Let $\tnab$ be a  $2$-symmetric Lorentzian connection on $(M,g)$ and $\xi$ a $\tnab$-recurrent vector field, i.e.~with $\tnab\xi=\theta\otimes \xi$.
Then $\tnab \d\theta$ is a $\tnab$-parallel tensor field.
If $\tnab$ is locally symmetric, then the $2$-form $\d\theta $ is parallel.
\end{lemma}
\begin{proof}
The proof follows from differentiating equation~\ref{recurv} to get,
\[(\tnab_X\tR)(Y,Z)\xi = (\nabla_X\d\theta)(Y,Z)\xi, \]
and 
\[(\tnab_V\tnab_X\tR)(Y,Z)\xi = (\tnab_V\tnab_X\d\theta)(Y,Z)\xi. \]
The local symmetry or $2$-symmetry then gives the results.
\end{proof}

\begin{proof}[{\bf Proof of Theorem~\ref{VFtheo}}]
If the holonomy algebra  of $\tnab$   is indecomposable and non-irreducible, it admits an invariant null line. Hence, on the universal cover, where the holonomy group is connected and therefore also admits an invariant null line, there is a $\tnab$-parallel null line bundle. 
On the universal cover this null line bundle is trivial and hence admits a global section $\xi$, which, since the null line bundle is parallel,  is a recurrent vector field, i.e.~$\tnab \xi=\theta\otimes \xi$.
By using Lemma~\ref{VFlem1},
for a proof by contradiction, we assume that there is a $p$ in the universal cover of $M$ such that $\d\theta|_p\not=0$. Hence, the holonomy algebra $\h$ of $\tnab$ at $p$ satisfies $\mathrm{pr}_{\R}(\h)\not=0$ and therefore is of type~1 or~3.

If $\tnab$ is locally symmetric, by Lemma~\ref{VFlem3} the $2$-form $\d \theta$ is parallel with respect to $\tnab$. Hence it is invariant under the holonomy $\h$. This means that  
the  skew endomorphism that is the metric dual to $\d \theta|_p$ is in the centraliser in $\so(T_pM) $  of $\h$. Since $\h$ is assumed to be indecomposable, non-irreducible, and with $m\ge 3$, it is of one of the four types in Theorem~\ref{bb-ike-theo}. Since $\mathrm{pr}_{\R}(\h)\not=0$, by Proposition~\ref{centprop} the centraliser is trivial, and hence $\d\theta|_p=0$, a contradiction.

If $\tnab$ is $2$-symmetric, then $\nabla \d\theta$ can be 
 identified via the metric with a  section $S$ of $T^*M\otimes \so(TM)$ that is parallel by Lemma~\ref{VFlem3}. If 
 $m\ge 4$, we can apply Theorem~\ref{algtheo} to $S|_p$ and obtain that $S|_p=0$ because of $\mathrm{pr}_{\R}(\h)\not=0$. Since $S$ is parallel, $S$ vanishes everywhere. This implies that $\d\theta$ is parallel and we can produce a contradiction as in the symmetric case.

Finally, the holonomy algebra of the universal cover is equal to the holonomy algebra of $(M,\tnab)$. Therefore it    annihilates a null vector and thus is contained in $\so(m-2) \ltimes \R^{m-2}$.
\end{proof}

\subsection{Lorentzian  connections with   parallel torsion and special holonomy}\label{LPTsec}
We say that a Lorentzian connection $\tnab$ has {\em parallel torsion} $T$ if $\tnab T=0$.
We will use  Theorem~\ref{algtheo}, to compare a Lorentzian connection with parallel torsion and indecomposable, non irreducible holonomy to the Levi-Civita connection. We obtain the following  result.

\begin{theorem}\label{vftheo}
Let $(M,g)$ a simply connected Lorentzian manifold of dimension $m\ge 4$ with Levi-Civita connection $\nabla$.  Assume that $\tnab$ is a Lorentzian connection with parallel torsion and  indecomposable, non-irreducible holonomy holonomy $\h$. Then $\tnab$ admits a parallel null vector field if and only if  the $\nabla$ admits a parallel null vector field, and both vector fields span the same null line bundle. Moreover, if $\tnab$ does not admit a parallel null vector field, then $\tnab=\nabla$.
\end{theorem}
%
%
%...
%\begin{lemma}\label{vflem}
%Let $(M,g)$ a simply connected Lorentzian manifold of dimension $m\ge 4$ and assume that $\tnab$ is an Lorentzian connection with parallel torsion, parallel null vector field $\xi$, and  indecomposable holonomy $\h$. 
%Then the Levi-Civita connection $\nabla$ admits a parallel null vector field that is a multiple of $\xi$.
%\end{lemma}
\begin{proof}
Let $\tnab$ be as in the assumptions and $\nabla$ the Levi-Civita connection of $(M,g)$.
We define  $S\in \Gamma(T^*M\otimes \g)$  by
\begin{equation}\label{S}S(X,Y)=\nabla_XY-\tnab_XY.\end{equation}
Since $\tnab$ has parallel torsion, $S$ is parallel for $\tnab$. 
Let $p\in M$ and $\h$ the holonomy algebra of $\tnab$ at $p$. Since $\h$ is indecomposable, non-irreducible, $M$ admits a $\tnab$-parallel null-line bundle. We can fix a basis of $T_pM$, so that $\h\in \p$. 

First assume that $\tnab$ does not admit a parallel null vector field, i.e.~that $\mathrm{pr}_{\R}(\h)\not=0$. For this case, in Theorem~\ref{algtheo} it was shown that $S|_p=0$. Since $p$ was arbitrary, this shows that $\tnab=\nabla$ and proves the last sentence of the Theorem.
It also applies that if $\nabla$ admits a parallel null vector field, also $\tnab $ does. 

To complete the proof, we have to assume that $\tnab$ admits a parallel null vector field $\xi$, ad we assume that the basis of $T_pM$ is chosen such that $e_-=\xi|_p$. We have that that $\h\subset \so(n)\ltimes \g_-$. For this case in Theorem~\ref{algtheo} it was shown that $S|_p$ has a very special form, in particular that $S|_p(X,e_-)=\e^+(X)\otimes \e_-$ for all $X\in TM$. Since $p$ was arbitrary, this implies 
that $S$ defines a function $\alpha$ by 
\begin{equation}\label{sxi} S(.,\xi)=\alpha \xi^\flat \otimes \xi,\end{equation}
where $\xi^\flat:=g(\xi,.)$.
Since $S$ and $\xi$  are $\tnab$-parallel,  $\alpha$ must in fact be {\em constant}.
Hence, $\xi$ is recurrent for $\nabla$ with $\theta=\alpha\,\xi^\flat$. However, 
\[\d\theta(X,Y)=\alpha\, \d \xi^\flat (X,Y)=-2 g(\xi, T(X,Y))=0,\]
by the  properties of the parallel torsion $T$ in Corollary~\ref{Talgcor}. Hence, by Lemma~\ref{VFlem1}, and since $M$ is simply connected, $\xi$ can be rescaled to a $\nabla$-parallel null vector field.
\end{proof}

Let   $\E:=\xi^\perp/\R\xi$   be  the screen bundle that is equipped with a $\tnabE$-parallel metric $\gE$, see Section~\ref{screensec}
Since $S$ is a section of $T^*M\otimes \mathcal P$, the endomorphism $S(X)$ preserves the vector field $\xi$ and its orthogonal bundle $\xi^\perp$. Hence,  it defines a section $S^{\E}$ of $T^*M\otimes \so(\E)$,
\[ S^{\E}(X)[Y]:=\left[ S(X,Y)\right],\quad \text{ for $X\in TM$ and $Y\in \xi^\perp$.} \]
If $\hat \xi$ is the parallel null vector field for $\nabla$, it has the same screen bundle, $\hat\xi^\perp/\R\hat\xi=\E$, but with a different connection $\nabE$.
With the definiton of $S$ in~(\ref{S}), $\nabE$ and $\tnabE$ are related by
\begin{equation}\label{SE}
S^\E(X)[Y]=\nabla^\E_X[Y]-\tnabE_X[Y].\end{equation}
This will be used in the proof of the following proposition. This proposition is not necessarily  needed for the proof of Theorem~\ref{maintheo}, 
but it is relevant to it and may be of independent interest.

\begin{proposition}\label{screenprop}Let $(M,g)$ be a Lorentzian manifold of dimension $m\ge 4$ with Levi-Civita connection $\nabla$ and $\tnab$ a Lorentzian connection with parallel torsion. Assume that $
\tnab$ admits a parallel null vector field $\xi$ and has  indecomposable holonomy.
There is a section $\sigma$ of $\so(\E)$ such that $S^{\E}=\xi^\flat\otimes \sigma$, and  $\sigma $ is parallel for $ \tnab^{\E}$ and $\nabla^{\E}$. Moreover, the curvatures and the holonomy algebras of $\tnab^\E$ and $\nabla^\E$ are equal.
\end{proposition}
\begin{proof}
By Theorem~\ref{algtheo}, the section $S^\E$ satisfies that  $S^{\E}(X)=0$, whenever $X\in \xi^\perp$.  Hence it is of the form $\xi^\flat\otimes \sigma$ for $\sigma $ a section of $\so(\E)$. 

To show that $\sigma$ is parallel, consider the connection 
$\widetilde{D}=\tnab\otimes \tnab^{\E}$ on $T^*M\otimes \so(\E)$.
A direct calculation shows that 
\[(\widetilde{D}_XS^\E)(Y)[Z] =\left[ (\tnab_XS)(Y,Z)\right],\]
so $\tnab S=0$ implies that $S^\E$ is parallel for $\widetilde{D}$. Then, since $\xi^\flat$ is parallel for $\tnab$, the section $\sigma$ is parallel for $\tnab^\E$, $\tnabE\sigma=0$.
Then, for $\nabE \sigma$ we get
\[
\nabE_X\sigma = [S^E(X),\sigma ] =\xi^\flat (X)[\sigma,\sigma]=0,\]
by the definition of $\sigma$.

Next we note that the curvatures of of $\nabE$ and $\tnabE$ are equal. This follows 
%\Tcom{This is a bit dodgy, but I wanted to keep the proof of Thm \ref{maintheo} self-contained}
from  Lemma~\ref{curvlem}, and from the same argument in the proof of Theorem~\ref{maintheo} in Section~\ref{LASsec}, that uses the algebraic Theorem ~\ref{algtheo} for $S$.

Finally, we will show that the holonomy algebras of $\nabE$ and $\tnabE$ are equal. Here we follow the proof in \cite[Proposition 2]{ErnstGalaev22}. For this, we define the bundle
\[\z(\sigma)=\{ \phi\in \so(\E)\mid [\sigma,\phi]=0\}. \]
Since $\sigma$ is parallel for both connections $\nabE$ and $\nabE$, it is a parallel for both induced connections on $\so(\E)$. With~(\ref{SE}) we have that  \[\xi^\perp\otimes \sigma=\nabE-\tnabE,\] so that the induced connections on $\z(\sigma)$ are equal. Moreover, since $\sigma$ is parallel, the curvature of $\tnabE$ (and also of $\nabE$, since they are equal) are sections of $\Lambda^2T^*M\otimes \z(\sigma)$. 
By the Ambrose-Singer holonomy theorem \cite{as}, the holonomy algebra of $\nabla$ on $\E$ is spanned by all endomorphism of the form
\[ P^{\so(E)}_\gamma(R^\E|_q (X,Y))= P_{\gamma}^{-1}\circ R|_q(X,Y) \circ  P_\gamma,\] where $\gamma $ is a curve from $p$ to $q$,  $X,Y\in T_qM$, and $P_\gamma$ and $P^{\so(E)}$ the parallel transports with respect to $\nabE$ in $\E$ and $\so(\E)$ , and in the same way for $\tnab$.
Since both curvatures are equal and contained in $\z(\sigma)$, and since $\z(\sigma)$ is parallel and both connections, and hence the parallel transports,  coincide on $\z(\sigma)$, this implies that the holonomy algebras are equal.
\end{proof}
In \cite{leistnerjdg} it was shown that the screen holonomy of the Levi-Civita connection of an indecomposable, non-irreducible Lorentzian manifold is a Riemannian holonomy algebra. Together with Proposition~\ref{screenprop} this yields the following result.
\begin{corollary}\label{screencor}Let $(M,g)$ be a Lorentzian manifold of dimension $\ge 4$ with  a Lorentzian connection  $\tnab$ with parallel torsion. Assume that $
\tnab$ admits a parallel null vector field $\xi$ and has and indecomposable holonomy.
Then the holonomy algebra of 
$ \tnab^{\E}$ is a Riemannian holonomy algebra.
\end{corollary}

\begin{remark}
In the special case where the torsion is not only parallel but also totally skew, the equality of the screen holonomies in Proposition~\ref{screenprop} was proved in \cite[Proposition 2]{ErnstGalaev22}. Also in this paper, several examples are given that show that, despite the equality of the screen holonomy,   the holonomy algebras of $\nabla$ and $\tnab$ do not have to be equal. In particular, in \cite[Example 1]{ErnstGalaev22} the connection $\tnab$ has indecomposable holonomy.
\end{remark}

%{\tt Idea:} Define $\F:=TM/\mathcal X$. Then $\tnab$ induces a connection $\tnab^{\F}$ on $\F$ by $\tnab^{\F}_X[Y]=[\tnab_XY]$. We have that $\E\subset \F$ is a $\tnab^{\F}$-parallel subbundle. Moreover, we get $S^{\F}$ as a section of $\gl(\F)$ such that $S^{\F}|_{\mathcal X^\perp}$ maps $\F$ into $\E$ and we define $\tau:=S^{\F}|_{\mathcal X^\perp}\in \Gamma ((\mathcal X^\perp)^*\otimes\gl(\F))$

\subsection{Lorentzian Ambrose--Singer connections and the proof of Theorem~\ref{maintheo}}\label{LASsec}
Now we consider Lorentzian Ambrose--Singer connections with special holonomy and prove a result about their curvature as well as Theorem~\ref{maintheo}. Both will be based on the algebraic results of the previous section and on Lemma~\ref{curvlem}.
\begin{theorem}\label{ASflattheo}
Let $(M,g)$ be a Lorentzian manifold of dimension $\ge 3$ with Ambrose--Singer connection $\tnab$ with indecomposable, non-irreducible holonomy. Then, on the universal cover,  $\tnab$ admits a parallel null vector field $\xi$, and the induced connection on the screen bundle is flat, or equivalently, $\tR(X,Y)=0$ for all $X,Y\in \xi^\perp$.
\end{theorem}
\begin{proof}
Without loss of generality, we can assume that $M$ is simply connected.
Let $T$ be the torsion and $\tR$ the curvature of $\tnab$, both $\tnab$-parallel.  
Since we have assumed that $M$ is simply connected, by Theorem~\ref{VFtheo}, there is  a  $\tnab$-parallel null vector field $\xi$.
If $M$ is of dimension~$3$, then the screen bundle has rank one, and hence, since the screen connection preserves a metric, is flat. Assume from now on that $\dim(M)\ge 4$.
Let $p\in \tM$ and   $\h$ be the holonomy algebra of $\tnab$ at $p$. The Lie algebra $\h$ is of types 2 or 4 from Theorem~\ref{bb-ike-theo}.
By Lemma~\ref{curvlem} this implies that $\tR|_p\in \mathcal R(T|_pM,\h, T|_p)$ as defined in Section~\ref{curvsec}.
Since $T$ is parallel, it is invariant under  the holonomy algebra $\h$,~i.e. $\h\cdot T|_p=0$.
by Corollary~\ref{Talgcor}, $T|_p$ satisfies condition~(\ref{Tcond}) of Theorem~\ref{curvtheo}. Since the curvature is also parallel, $\tR|_p$ is annihilated by $\h$. Therefore, by Corollary~\ref{curvtheo2}, $\tR|_p(u,v)\in \g_-$ for all $u,v\in T_pM$. Hence, the curvature of the screen bundle vanishes at $p$. Since $p$ was arbitrary, the screen bundle  is flat.  
\end{proof}

This  yields an independent proof of the classification result by Cahen and Wallach in \cite{cahen-wallach70}.

\begin{corollary}
Let $(M,g)$ be a locally symmetric Lorentzian manifold with indecomposable, non-irreducible holonomy. Then $(M,g)$ is locally isometric to a Cahen--Wallach space as defined in Section~\ref{homplanesec}.
\end{corollary}
\begin{proof} If $(M,g)$ is loally symmetric, the Levi-Civita connection is an Ambrose--Singer connection, and , by
Theorems~\ref{VFtheo} and~\ref{ASflattheo}, it satisfies the curvature condition of a pp-wave as defined in Section~\ref{homplanesec}. Hence $g$ is locally of the form~(\ref{pploc}).  Local symmetry  then yields that
$h(t,x^1,\ldots, x^n)= x^i Q_{ij} x^j$ for a constant, symmetric non degenerate matrix $Q_{ij}$. As a consequence, $(M,g)$ is locally isometric to a Cahen--Wallach space.
\end{proof}

We will now  prove Theorem~\ref{maintheo}. 
This will require a comparison of  a special holonomy  Ambrose-Singer connection $\tnab$ with the Levi-Civita connection of $(M,g)$ as in Lemma~\ref{curvlem}.

\begin{proof}[Proof of Theorem~\ref{maintheo}]
Since the statement of Theorem~\ref{maintheo} is about the universal cover, we assume that
 $(M,g)$ is a simply connected Lorentzian manifold with Levi-Civita connection $\nabla$. Let $\tnab$ be an  Ambrose--Singer connection 
with  $S(X,Y)=\nabla_XY-\tnab_XY$ and $T(X,Y)=-S(X,Y)+S(Y,X)$, the torsion of $\tnab$,  both  $\tnab$-parallel. 
Assume that $\tnab$ 
 has an indecomposable, non-irreducible holonomy algebra. Since $\dim M=m\ge 3$,  by Theorem~\ref{VFtheo}, $\tnab$ admits a parallel null vector field $\xi$.  By Theorem~\ref{vftheo}  also the Levi-Civita connection admits a parallel null vector field $\hat\xi =f\xi$ for a function $f$, so it remains to establish the curvature conditions of a plane wave. First we establish that $(M,g)$ is a pp-wave, i.e.
   the flatness of the connection that is induced on  the screen bundle $\mathcal E=\xi^\perp/\xi$  from the Levi-Civita connection $\nabla$. By Theorem~\ref{ASflattheo}, the connection induced on $\mathcal E$ from $\tnab$ is flat, or equivalently, that 
 \[\tR(V,X,Y,Z)=g(\tR(V,X)Y,Z)=0,\quad\text{ for all $V\in TM$ and $X,Y,Z\in \xi^\perp$.}\] 
 Hence, by Lemma~\ref{curvlem} we have for the curvature of the Levi-Civita connection $\nabla$ that
 \[ R(V,X,Y,Z)=g(S(T(V,X),Y),Z)+g([ S(V),S(X)]Y,Z),\]
 for all $V\in TM$ and $X,Y,Z\in \xi^\perp$. We will now show that at any point $p\in M$ both terms on the right-hand-side of this equation vanish, i.e. we will show that
 \begin{equation}\label{ST}
 S(T(V,X),Y)\in \R\cdot\xi,\quad [ S(V),S(X)]Y\in \R\cdot\xi,\ \text{
 for all $V\in TM$ and $X,Y\in \xi^\perp$.}\end{equation}
  We fix a Witt basis $(e_-,e_1,\ldots, e_n, e_+)$ of $T_pM=V_-\+V_0\+V_+$ as in Section~\ref{algsec} and  such that $\xi|_p=e_-$ and $\xi^\perp|_p=\mathrm{span}(e_-,e_1, \ldots, e_n)$. Then, if $\h$ is the holonomy algebra of $\tnab$ at $p$, we have that $\h\subseteq\so(V_0)\ltimes V_0$ is indecomposable. Since $S$ and $T$ are parallel, $S|_p$ and $T|_p$  are annihilated by $\h$. Now let $V\in T_pM$ and $X,Y,Z\in \xi^\perp|_p$. By Corollary~\ref{Talgcor}, $T(V,X)\in \xi^\perp|_p$, so that by Theorem~\ref{algtheo}, $S(T(V,X),Y)\in \R\xi|_p$. This implies that  $g(S(T(V,X)Y,Z)=0$. Moreover, since $S(X)\in \g_-$, we have that $[S(V),S(X)]\in \g_-$, so that again $[S(V),S(X)]Y\in \R\cdot \xi$ and hence $g([S(v),S(x)]y,z)=0$. Since $p$ was arbitrary, this shows properties~(\ref{ST}) and hence that $R(V,X,Y,Z)=0$
 for all $V\in TM$ and $X,Y,Z\in \xi^\perp$, i.e.~$(M,g)$ satisfies the curvature condition of a pp-wave. 
 
 It remains to show, when $n\ge 4$, that $(M,g)$ is a plane wave, i.e.~that $\nabla_XR=0$ for all $X\in \xi^\perp$. By Lemma~\ref{curvlem} we have that 
 \[0=\tnab_X R=\nabla_X R -S(X)\cdot R.\]
Again working at an arbitrary point $p$ with a Witt basis of $T_pM$, by Theorem~\ref{algtheo}, we have that $S(X)\in \g_-$ for all $X\in \xi^\perp|_p$. Since $(M,g)$ is a pp-wave,  $R|_p\in V^+\wedge V^0 \otimes \g_-$, which implies that for all $X,Y\in \xi^\perp|_p $ and $V\in T_pM$ that
\[(S(X)\cdot R)(Y,V)=[S(X),R(Y,V)]-R(S(X,Y),V)- R(Y,S(X,V))=0,\]
since $\g_-$ is abelian, $S(X,V)\in V_0$ and $S(X,Y)\in \R\cdot \xi|_p$. This shows that $S(X)\cdot R=0$ and completes the proof that $(M,g)$ is a plane wave. By the results cited in Section~\ref{review} it is also locally homogeneous.
\end{proof}

\subsection{The $3$-dimensional case}\label{dim3sec}
The main purpose of this section is to show that the dimension restriction to $m>3$ in 
Theorem~\ref{maintheo} and Conjecture~\ref{conj} is sharp. But first we make an observation that holds in any dimension.
\begin{lemma}\label{hg-lemma}
Let $\h=\g_-\subset \so(1,n+1)$ and $R\in \Lambda^2 V^*\otimes \h$ such that $\h\cdot R=0$. Then $e_-\hook R=0$.
\end{lemma}
\begin{proof}
Let $\h=\g_-=\Span( \overline{e}_1, \ldots, \overline{e}_n)$ for $n=m-2\ge 1$ (with the conventions of Section~\ref{algsec1}), and $R\in \Lambda^2V^*\otimes \g_-$. Since $\g_-$ is abelian, $\g_-\cdot R=0$ implies
\[ 0=(\overline{e}_i\cdot R) (e_-,e_+)=-R(e_-,e_i),\qquad   0=(\overline{e}_i\cdot R)(e_i,e_+)=R(e_-,e_+),\]
so that $R(e_-,v)=0$ for all $v\in V$. Hence, $e_-\hook R=0$.
\end{proof}

From now on assume that $\dim(V)=3$.
First we will see that Theorem~\ref{algtheo} is false in dimension $3$.
\begin{lemma}
\label{example1}
Let $V=\span(e_-,e_1,e_+)$ and  $\h=\R\cdot X=\g_-\subset \so(1,2)$, where 
\[ X:= \overline{e}_1= \begin{pmatrix}
0&-1&0\\ 0&0&1 \\0&0&0\end{pmatrix}.\]
If  $S\in V^*\otimes \so(1,2)$ such that $\h\cdot S=0$, then
\[S(e_-)=a X,
\quad 
S(e_1)= 
\begin{pmatrix}
a&-b&0\\ 0&0&b \\0&0&-a\end{pmatrix} 
%\not\in \g_-
,
\quad
S(e_+)= 
\begin{pmatrix}
-b&-c&0\\ -a&0&c \\0&a&b\end{pmatrix}
%\not\in \p
,
\]
for real numbers $a$, $b$ and $c$. 
In particular, the maximal trivial submodule of $V^*\otimes \g$ on which $\h$ acts trivial is of dimension $3$. 
\end{lemma}

\begin{proof}
Since $\h=\R\cdot X$
%, i.e.
%\[X=\begin{pmatrix}
%0&-1&0\\ 0&0&1 \\0&0&0\end{pmatrix}.\]
%Define $S\in V^*\otimes \g$ by 
%\[S(e_-)=a X,
%\quad 
%S(e_1)= 
%\begin{pmatrix}
%a&-b&0\\ 0&0&b \\0&0&-a\end{pmatrix} \not\in \g_-,
%\quad
%S(e_+)= 
%\begin{pmatrix}
%-b&-c&0\\ -a&0&c \\0&a&b\end{pmatrix}\not\in \p,
%\]
we have to verify that $X\cdot S=0$. Indeed, 
\[ [X,S(e_-)]=0,\quad [X,S(e_1)] =-aX,\quad   [X,S(e_+)] =S(e_1),\]
so that 
\[ X\cdot S (e_1)=[X, S(e_1)] -S(Xe_1)= -aX +S(e_-)=0,\]
and 
\[ X\cdot S (e_+)=[X, S(e_+)] -S(Xe_+)= S(e_1) -S(e_1)=0.\]
Conversely, it is easy to check that the above $S$ is the most general element  in the submodule of  $V^*\otimes \g$ on which $\h$ acts trivially. This shows that this submodule is of dimension $3$.
\end{proof}
This has the following consequence for trivial submodules of the torsion module. If $T(u,v)=-S(u,v)+S(v,u)$ for an $S$ as in Lemma~\ref{example1}, we get
\[T(e_-,e_1)=2a e_-,\quad T(e_-,e_+)=-2a e_1 -b e_-,\quad
T(e_1,e_+)=2a e_+-be_1-ce_-.\]
 Lemma~\ref{example1} shows why Theorem~\ref{algtheo} does not cover the case $\dim(V)=3$, when $V=\R^{1,2}$. As a consequence,  a torsion tensor $T$ with $\h\cdot T=0$ does not necessarily satisfy the assumptions of Theorem~\ref{curvtheo} and thus Corollary~\ref{curvtheo2}.

Now, let  $(M,g)$ be a $3$-dimensional Lorentzian manifolds with Levi-Civita connection $\nabla$ and curvature tensor $R$. We 
 assume that $\tnab$ is an Ambrose--Singer connection on $(M,g)$ that is defined by $S$ as in Lemma~\ref{example1}. Since $S$ is parallel, it is sufficient to specify $S$ at a point~$o$.
  Let $(e_-,e_1,e_+)$ be a Witt basis of $V=T_oM$, $\h=\R\cdot X=\g_-$ the holonomy algebra of $\tnab$ at $o$ and so that $\tR(e_1,e_+)|_o=X=\overline{e}_1$. 
 By  Lemma~\ref{hg-lemma} every $\tR\in \Lambda^2 V^*\otimes \h$ such that $\h\cdot \tR=0$ satisfies $\e_-\hook \tR=0$. Because of this and of $m=3$, we have that $\tR\in V^0\wedge V^+\otimes \h$, which is of dimension~$1$. Hence, $\tR$ is equal to  (a multiple of) the curvature tensor of a $3$-dimensional Cahen--Wallach space, and therefore
$\tR$ satisfies the Bianchi symmetry $\tR(u,v)w+\tR(v,w)u+\tR(w,u)v=0$. We arrange the basis such that $\tR(e_1,e_+)=\overline{e}_1$. 
As a consequence, by Lemma~\ref{curvlem}, its torsion satisfies,
\[T(T(u,v),w)+ T(T(v,w),u)+T(T(w,u),v)=0\] for all $u,v,w\in TM$.
For an $S$ as in  Lemma~\ref{example1}, we evaluate this equation for $u=e_-$, $v=e_1$ and $w=e_+$ to get
\[0=T(T(e_-,e_1),e_+)+ T(T(e_1,e_+),e_-)+T(T(+,e_-),e_1)=4\, ab\,e_-.\] 
Hence, $S$ has to lie in the subset of $W$ that is defined by the union of the hyperplanes $W_1=\{a=0\}$ and $W_2=\{b=0\}$. 
If $S\in W_1$, then $S$ satisfies the conclusions of Theorem~\ref{algtheo}, so one can proceed in the  proof of Theorem~\ref{maintheo} as in the case $n\ge 4$. Hence, we will focus on the case $a\not=0$, i.e. 
 $S\in W_2$, so that $b=0$, which we will assume from now on. Then 
\begin{equation}\label{counterS}
S(e_-)=a X,
\quad 
S(e_1)= 
a I,
\quad
S(e_+)= 
%\begin{pmatrix}
%0&-c&0\\ -a&0&c \\0&a&0\end{pmatrix} =
cX+aX^\top,\quad \text{ where } I= \begin{pmatrix}
1&0&0\\ 0&0&0 \\0&0&-1\end{pmatrix},
\end{equation}
$X=\overline{e}_1$, and real numbers $a$ and $c$. It is $I=-E$, w of $\so(1,2)$  satisfies $[I,X]=X$ and $[I,X^\top]=-X^\top$ and $[X,X^\top]=I$. 

We can use Lemma~\ref{curvlem} to obtain the curvature of the Levi-Civita connection as follows
\begin{eqnarray*}
R(e_-,e_1)&=& S(T(e_-,e_1))+[S(e_-),S(e_1)] =a^2 X,\\
R(e_-,e_+)&=& S(T(e_-,e_+))+[S(e_-),S(e_+)] =-a^2 I,\\
R(e_1,e_+)&=& X+S(T(e_1,e_+))+[S(e_1),S(e_+)] =
(1+2ac)X
+a^2X^\top
\end{eqnarray*}
This shows that when $a\not=0$, then the holonomy algebra of the Levi-Civita connection is of dimension $3$ and hence equal to $\so(1,2)$. Hence $(M,g)$ cannot be plane wave nor have special holonomy. Furthermore,
the Ricci tensor, in the basis $e_-,e_1,e_+$ is given as 
\[
\mathrm{Ric}=\begin{pmatrix}
0&0&-2a^2
\\
0&-2a^2&0
\\
-2a^2&0& 1+2ac,
\end{pmatrix}\]
so that
 the scalar curvature is a {\em negative} multiple of $a^2$ and the metric is Einstein if and only if $1+2ac=0$. Using that 
$\nabla_VR=S(V)\cdot R$ from Lemma~\ref{curvlem}, one further computes that the only non-vanishing derivatives of $R$ are 
\[
\nabla_{e_1}R(e_1,e_+)=\nabla_{e_+}R(e_-,e_+)=a (1+2ac) X, \quad \nabla_{e_+}R(e_1,e_+)=-a (1+2ac) I.\]
Hence,  $(M,g)$ is locally isometric to anti-de~Sitter space if and only if $1+2ac=0$.

In order to describe the locally homogeneous space in the case $a\not=0$, we consider the infinitesimal model defined by the vector space $V=\R^{1,2}$, the inner product defined by the Witt basis $(e_-,e_1,e_+)$ and the pair $\tR$ and $T$, see for example \cite[Section 2.4]{CalvarusoCastrillon-Lopez19}. For this set $\h=\R\cdot X\subset \g$, and 
 $\g=\h\+V$ is the transvection algebra,  with Lie brackets 
\[ [X,v]=Xv,\quad [u,v]:= -\tR(u,v)-T(u,v).\]
Hence, the Lie brackets in $\g$ are
\[
[X,e_-]=0,\qquad [X,e_1]=-e_-,\qquad [X,e_+]=e_1\]
and
\[ [e_-,e_1] =-2a e_-,\qquad [e_-,e_+]= 2a e_1,\qquad [e_1,e_+]=-X-2ae_++ce_-.\]
The derived Lie algebra is $\g'=\mathrm{span}(e_-,e_1,E_+)$, where we set $
E_+:= \tfrac{1}{2a}X+ e_+$,
so that  $\< e_-,E_+\>=1$,  
\[ [e_-,e_1] =-2a e_-\qquad [e_-,E_+] =2ae_1,\qquad [e_1,E_+] =-2aE_++ \tfrac{(2ac+1)}{2a}e_- \]
and therefore $\g'=[\g',\g']$. Because of  $a\not=0$, $\g'$ is locally transitive. A check of the Killing form reveals that it is not definite, so that $\g'$ is isomorphic to $\sl(2,\R)$.
%{\tt (see 3dim.maple)}. 
Together with 
Theorem~\ref{VFtheo} get the following conclusion.

\begin{proposition}\label{finalprop}
If  a simply-connected Lorentzian manifold $(M,g)$ of dimension $3$ admits an Ambrose--Singer connection with indecomposable, non-irreducible holonomy, then $(M,g)$ is a plane-wave or locally isometric to a left-invariant metric on $\widetilde{\SL}(2,\R)$ with holonomy algebra $\so(1,2)$ and negative scalar curvature.
\end{proposition}
Finally, we will show that a manifold with an Ambrose--Singer connection that satisfies the assumption of Proposition~\ref{finalprop} does indeed exist. We will give an example of a left-invariant metric on the universal cover of $\SL(2,\R)$ that is not a pp-wave metric but admits and  a non-flat Ambrose--Singer connection with a parallel null vector field, and hence with indecomposable holonomy algebra.
\begin{example}\label{counterex}
Consider $\m:=\R^3$ with Lie bracket defined by
\[ [e_-,e_1] =-2a e_-\qquad [e_-,e_+] =2ae_1,\qquad [e_1,e_+] =-2ae_++ \tfrac{(2ac+1)}{2a}e_- ,\]
where  $e_-,\e_1,\e_+$ is a basis of $\m$.
This Lie algebra satisfies $\m'=\m$ and the Killing form is not definite, so it must be isomorphic to $\sl(2,\R)$. Let $\<.,.\>$ be the Minkowski inner product on $\m$ that is defined by the requirement that $(e_-,e_1,e_+)$ is a Witt basis for it. This inner product defines a left-invariant metric $g$ on the simply-connected Lie group $M$ that has $\m$ as Lie algebra.  The Lie group $M$ is isomorphic to $\widetilde{\SL}(2,\R)$, the universal cover of $\SL(2,\R)$. We identify 
 $g$ with $\<.,.\>$ and $e_-,e_1,e_+$ with the left-invariant vector fields on $M$, which form a Witt basis at every tangent space. The Koszul formula for the Levi-Civita connection gives
\[2\< \nabla_ue_-,v\>=\< [u,e_-],v\> + \< [v,e_-],u\>+ \< [v,u],e_-\>,\]
for left-invariant vector fields $u$ and $v$.
Using this, a computations shows that 
\[
\nabla_{e_-}e_-=0,\qquad \nabla_{e_1}e_-=ae_-,\qquad \nabla_{e_+}e_-=-ae_1.\]
Now we consider the left invariant homogeneous structure given by $S\in \m^*\otimes \so(\m)$ as in~(\ref{counterS}). One can easily check that this is indeed a homogeneous structure, by showing that $\tnab_XY =\nabla_XY-S(X,Y)$ is an Ambrose--Singer connection, i.e. that $\tnab \tR=0$ and $\tnab S=0$. The formulae for $\nabla$ and $S$ then show that the left invariant null vector field corresponding to $e_-$ is parallel for $\tnab$, since
\[S(e_-,e_-)=0,\qquad S(e_1,e_-)=ae_-,\qquad S(e_+,e_-)=-a e_1.\]
It remains to show that $(M,g)$ is not flat. A direct computation reveals that 
$\tR(e_1,e_+)=X$, so that the holonomy algebra is equal to $\R\cdot X$ and hence indecomposable. 

Note that this is also a counterexample to Conjecture~\ref{conj} in dimension  $m=3$. Indeed, constructing the infinitesimal model $(\g,\h)$ with $\h:=\R\cdot X$ as above, then 
\[\g\ =\ \R\cdot X\ltimes \m\ \simeq\   \R\ltimes \sl(2,\R).\] This yields that $M=G/H$ with $ G$ the simply connected Lie group with Lie algebra $\g$ and $H$ the connected subgroup with indecomposable isotropy  algebra $\h$.
\end{example}

\bibliographystyle{abbrv}
\bibliography{GEOBIB}
\end{document}